\documentclass[fleqn,reqno]{amsart}
\usepackage{url}
\usepackage{amssymb,amsmath,amssymb,color,cite}
\usepackage{amsmath,amsthm,amssymb, amstext,}
\usepackage{graphicx}
\usepackage{multirow}
\newtheorem{theorem}{Theorem}[section]

\newtheorem{corollary}{Corollary}[section]

\newtheorem{open problem}{\sc \bf Open problem}[section]
\newtheorem{remark}{Remark}[section]

\DeclareMathOperator{\sech}{sech}
\newtheorem{open}{Open problem}[section]
\numberwithin{equation}{section} 
\begin{document}

\title[]{Series representations of the remainders in the expansions for certain trigonometric and hyperbolic functions with applications}

\author[{ C.-P. Chen }]{ Chao-Ping Chen$^{*}$}

\address{C.-P. Chen: School of Mathematics and Informatics, Henan Polytechnic University, Jiaozuo City 454000, Henan Province, China}
 \email{chenchaoping@sohu.com}

\author[{ R.B. Paris }]{Richard B. Paris}

\address{R.B. Paris: Division of Computing and Mathematics,\\
 University of Abertay, Dundee, DD1 1HG, UK}
 \email{R.Paris@abertay.ac.uk}

\thanks{*Corresponding Author}



\begin{abstract}
In this paper, we   present  series
representations of the remainders in the expansions for certain trigonometric and hyperbolic functions.
 By using the obtained
results, we establish some inequalities for trigonometric and hyperbolic functions.
\end{abstract}


\maketitle

\section{Introduction }

 The Bernoulli numbers $B_n$ and Euler numbers
$E_n$ are defined, respectively, by the following generating functions:
\begin{equation*}\label{generalizedBernoullipolynomials}
\frac{t}{e^{t}-1}=\sum_{n=0}^{\infty}B_{n}\frac{t^n}{n!}\quad
(|t|<2\pi)\quad \text{and}\quad
\sec t=\sum_{n=0}^{\infty}E_{n}\frac{t^n}{n!}\quad
(|t|<\pi).
\end{equation*}

Recently, Chen and Paris \cite{Chen-Paris-submission} presented  series
representations of the remainders in the expansions for $2/(e^t+1)$, $\sech t$ and $\coth t$.
 For example, these authors proved that for  $t > 0$ and $N\in\mathbb{N}:=\{1, 2, \ldots\}$,
\begin{align*}
\sech t=\sum_{j=0}^{N-1}\frac{E_{2j}}{(2j)!}t^{2j}+R_N(t)
\end{align*}
with
\begin{align*}
R_N(t)=\frac{(-1)^{N}2t^{2N}}{\pi^{2N-1}}\sum_{k=0}^{\infty}\frac{(-1)^{k}}{(k+\frac{1}{2})^{2N-1}\Big(t^2+\pi^2(k+\frac{1}{2})^2\Big)},
\end{align*}
and
\begin{align*}
\sech t=\sum_{j=0}^{N-1}\frac{E_{2j}}{(2j)!}t^{2j}+\Theta(t,
N)\frac{E_{2N}}{(2N)!}t^{2N}
\end{align*}
with a suitable $0 < \Theta(t, N) < 1$.  By using the obtained
results, these authors deduced some inequalities and completely monotonic
functions associated with the ratio of gamma functions.

This paper is a continuation of our earlier work \cite{Chen-Paris-submission}. We here present  series
representations of the remainders in the expansions  for certain trigonometric and hyperbolic functions.
 By using the obtained
results, we establish some inequalities for trigonometric and hyperbolic functions.
\vskip 8mm

\section{Series representations of the remainders}
\begin{theorem}\label{Thm1-remainder-tan}
Let $N\geq0$ be an integer. Then for $|t|<\pi/2$, we have
\begin{align}\label{Thm1-tan-series-representation}
\tan t =\sum_{j=1}^N\frac{2^{2j}(2^{2j}-1){|B_{2j}|}}{(2j)!}t^{2j-1}+\vartheta_N(t),
\end{align}
where
\begin{align}\label{Thm1-remainder-tan-series-representation}
\vartheta_N(t)=\frac{2^{2N+3}t^{2N+1}}{\pi^{2N}}\sum_{k=1}^{\infty}\frac{1}{(2k-1)^{2N}\Big(\pi^2(2k-1)^2-4t^2\Big)}.
\end{align}
Here, and throughout this paper, an empty sum is understood to be zero.
\end{theorem}

\begin{proof}
It follows from \cite[p. 44]{Gradshteyn2000} that
\begin{equation*}\label{Euler-constantSm-inequality1}
\tan\frac{\pi x}{2}=\frac{4x}{\pi}\sum_{k=1}^{\infty}\frac{1}{(2k-1)^{2}-x^2}.
\end{equation*}
Replacement of $x$ by $2t/\pi$ yields
\begin{equation}\label{tan-expansion}
\tan t=\sum_{k=1}^{\infty}\frac{8t}{\pi^2(2k-1)^{2}-4t^2},
\end{equation}
which can be written as
\begin{equation}\label{tan-expansion-re}
\tan t=\frac{8t}{\pi^2}\sum_{k=1}^{\infty}\frac{1}{(2k-1)^2\left(1-\left(\frac{2t}{\pi(2k-1)}\right)^2\right)}.
\end{equation}

Using the following identities:
\begin{equation}\label{identity1/(1-q)}
\frac{1}{1-q}=\sum_{j=0}^{N-1}q^{j}+\frac{q^N}{1-q}\qquad
(q\not=1)
\end{equation}
and
\begin{equation}\label{odd-series-Expansion}
\sum_{k=1}^{\infty}\frac{1}{(2k-1)^{2n}}=\frac{(2^{2n}-1)\pi^{2n}|B_{2n}|}{2\cdot(2n)!}
\end{equation}
(see \cite[p. 8]{Gradshteyn2000}), we obtain from \eqref{tan-expansion-re} that
\begin{align*}
\tan t&=\frac{8t}{\pi^2}\sum_{k=1}^{\infty}\frac{1}{(2k-1)^2}\left(\sum_{j=0}^{N-1}\left(\frac{2t}{\pi(2k-1)}\right)^{2j}+\frac{\left(\frac{2t}{\pi(2k-1)}\right)^{2N}}{1-\left(\frac{2t}{\pi(2k-1)}\right)^2}\right)\\
&=\sum_{j=1}^N\frac{2^{2j}(2^{2j}-1){|B_{2j}|}}{(2j)!}t^{2j-1}+\vartheta_N(t),
\end{align*}
where
\begin{align*}
\vartheta_N(t)=\frac{2^{2N+3}t^{2N+1}}{\pi^{2N}}\sum_{k=1}^{\infty}\frac{1}{(2k-1)^{2N}\Big(\pi^2(2k-1)^2-4t^2\Big)}.
\end{align*}
 The proof of Theorem \ref{Thm1-remainder-tan} is complete.
\end{proof}

Becker and Stark \cite{Becker133--138} showed that for $0<x<\pi/2$,
\begin{equation}\label{Becker-Stark-ineq}
\frac{8}{\pi^2-4x^2}<\frac{\tan x}{x}<\frac{\pi^2}{\pi^2-4x^2}.
\end{equation}
The constant $8$ and $\pi^2$ are the best possible.
The Becker--Stark inequality \eqref{Becker-Stark-ineq} has
attracted much interest of many mathematicians and has motivated a
large number of research papers (cf.
\cite{Banjac-Makragic-Malesevic,Chen-Cheung2011,Debnath-Mortici-Zhu207--215,Ge2012,NishizawaJIA2015,Sun-Zhu563--567,ZhuID931275,ZhuID838740,Zhu401--406}
and the references cited therein). For example,  Banjac et al.  \cite[Theorem 2.7]{Banjac-Makragic-Malesevic} proved in 2015 that for $0<x<\pi/2$,
\begin{align}\label{Thm1-tan-Banjac}
&\frac{\pi^2+\left(\frac{\pi^2}{3}-4\right)x^2+\left(\frac{\pi^2}{18}-\frac{2}{3}\right)x^4}{\pi^2-4x^2}<\frac{\tan x}{x}<\frac{\pi^2-\frac{\pi^2}{16}x^2+\frac{1}{2}x^4-\frac{1}{\pi^2}x^6}{\pi^2-4x^2}.
\end{align}
There is no strict comparison between  the two lower bounds in \eqref{Becker-Stark-ineq} and \eqref{Thm1-tan-Banjac}. The upper bound in \eqref{Thm1-tan-Banjac} is sharper than that in \eqref{Becker-Stark-ineq}.

Write \eqref{Thm1-tan-series-representation} as
\begin{align}\label{Thm1-tan-series-representationre}
\frac{\tan t}{t} &=\sum_{j=1}^N\frac{2^{2j}(2^{2j}-1){|B_{2j}|}}{(2j)!}t^{2j-2}\nonumber\\
&\quad+\frac{2^{2N+3}t^{2N}}{\pi^{2N}}\left\{\frac{1}{\pi^2-4t^2}+\sum_{k=2}^{\infty}\frac{1}{(2k-1)^{2N}\Big(\pi^2(2k-1)^2-4t^2\Big)}\right\}.
\end{align}
Noting that the function
\begin{align*}
F(t):=\sum_{k=2}^{\infty}\frac{1}{(2k-1)^{2N}\Big(\pi^2(2k-1)^2-4t^2\Big)}
\end{align*}
is strictly increasing for $0<t<\pi/2$, we then obtain from \eqref{Thm1-tan-series-representationre} that for $0<t<\pi/2$,
\begin{align}\label{Thm1-tan-series-representationreobtain}
\frac{2^{2N+3}t^{2N}}{\pi^{2N+2}}\sum_{k=2}^{\infty}\frac{1}{(2k-1)^{2N+2}}&<\frac{\tan t}{t}-\sum_{j=1}^N\frac{2^{2j}(2^{2j}-1){|B_{2j}|}}{(2j)!}t^{2j-2}-\frac{2^{2N+3}t^{2N}}{\pi^{2N}(\pi^2-4t^2)}\nonumber\\
&<\frac{2^{2N+1}t^{2N}}{\pi^{2N+2}}\sum_{k=2}^{\infty}\frac{1}{k(k-1)(2k-1)^{2N}}.
\end{align}
Direct computations yield
\begin{align*}
\sum_{k=2}^{\infty}\frac{1}{(2k-1)^{4}}=\frac{\pi^4}{96}-1,  \quad \sum_{k=2}^{\infty}\frac{1}{k(k-1)(2k-1)^{2}}=5-\frac{\pi^2}{2}.
\end{align*}
The choice $N=1$ in \eqref{Thm1-tan-series-representationreobtain} yields
\begin{align*}
\frac{32t^{2}}{\pi^{4}}\left(\frac{\pi^4}{96}-1\right)&<\frac{\tan t}{t}-1-\frac{32t^{2}}{\pi^{2}(\pi^2-4t^2)}<\frac{8t^{2}}{\pi^{4}}\left(5-\frac{\pi^2}{2}\right),\qquad
0<t<\frac{\pi}{2},
\end{align*}
which can be rearranged for $0<x<\pi/2$ as
\begin{align}\label{Thm1-tan-series-representationreyieldsN=1re}
&\frac{\pi^2+\frac{\pi^2-12}{3}x^2+\frac{384-4\pi^4}{3\pi^4}x^4}{\pi^2-4x^2}<\frac{\tan x}{x}<\frac{\pi^2+\frac{72-8\pi^2}{\pi^2}x^2+\frac{16\pi^2-160}{\pi^4}x^4}{\pi^2-4x^2}.
\end{align}
The inequality \eqref{Thm1-tan-series-representationreyieldsN=1re} improves the inequalities \eqref{Becker-Stark-ineq} and \eqref{Thm1-tan-Banjac}.

\begin{theorem}\label{Thm2-remainder-tanh}
Let $N\geq0$ be an integer. Then for all $t\in\mathbb{R}$, we have
\begin{align}\label{Thm2-remainder-tanh-expansion}
\tanh t=\sum_{j=1}^{N}\frac{2^{2j}(2^{2j}-1){B_{2j}}}{(2j)!}t^{2j-1}+\tau_N(t),
\end{align}
where
\begin{align}\label{Thm2-remainder-tanh-tauN}
\tau_N(t)=(-1)^{N}\frac{2^{2N+3}t^{2N+1}}{\pi^{2N}}\sum_{k=1}^{\infty}\frac{1}{(2k-1)^{2N}\Big(\pi^2(2k-1)^2+4t^2\Big)},
\end{align}
and
\begin{align}\label{Thm2-remainder-tanh-xi}
\tanh t=\sum_{j=1}^{N}\frac{2^{2j}(2^{2j}-1){B_{2j}}}{(2j)!}t^{2j-1}+\xi(t,N)\,\frac{2^{2N+2}(2^{2N+2}-1) B_{2N+2}}{(2N+2)!}\,t^{2N+1},
\end{align}
where $0<\xi(t, N)<1$.
\end{theorem}

\begin{proof}
It follows from \cite[p. 44]{Gradshteyn2000} that
\begin{equation}\label{tanh-expansion}
\tanh\frac{\pi x}{2}=\frac{4x}{\pi}\sum_{k=1}^{\infty}\frac{1}{(2k-1)^2+x^2}.
\end{equation}
Replacement of $x$ by $2t/\pi$ yields
\begin{equation}\label{tanh-expansion-re}
\tanh t=\sum_{k=1}^{\infty}\frac{8t}{\pi^2(2k-1)^{2}+4t^2}=\frac{8t}{\pi^2}\sum_{k=1}^{\infty}\frac{1}{(2k-1)^2\left(1+\left(\frac{2t}{\pi(2k-1)}\right)^2\right)}.
\end{equation}
Using the following identity:
\begin{equation}\label{identity1/(1+z)}
\frac{1}{1+q}=\sum_{j=0}^{N-1}(-1)^{j}q^{j}+(-1)^{N}\frac{q^N}{1+q}\qquad
(q\not=-1)
\end{equation}
and \eqref{odd-series-Expansion}, we obtain from \eqref{tanh-expansion-re} that
\begin{align*}
\tanh t&=\frac{8t}{\pi^2}\sum_{k=1}^{\infty}\frac{1}{(2k-1)^2}\left(\sum_{j=0}^{N-1}(-1)^{j}\left(\frac{2t}{\pi(2k-1)}\right)^{2j}+(-1)^{N}\frac{\left(\frac{2t}{\pi(2k-1)}\right)^{2N}}{1+\left(\frac{2t}{\pi(2k-1)}\right)^2}\right)\\
&=\sum_{j=1}^{N}\frac{2^{2j}(2^{2j}-1){B_{2j}}}{(2j)!}t^{2j-1}+\tau_N(t),
\end{align*}
where
\begin{align*}
\tau_N(t)=(-1)^{N}\frac{2^{2N+3}t^{2N+1}}{\pi^{2N}}\sum_{k=1}^{\infty}\frac{1}{(2k-1)^{2N}\Big(\pi^2(2k-1)^2+4t^2\Big)}.
\end{align*}

Noting that \eqref{odd-series-Expansion} holds, we can rewrite $\tau_N(t)$ as
\begin{align*}
\tau_N(t)=\xi(t,N)\,\frac{2^{2N+2}(2^{2N+2}-1) B_{2N+2}}{(2N+2)!}\,t^{2N+1},
\end{align*}
where
\begin{align*}
\xi(t,N):=\frac{g(t)}{g(0)},\qquad g(t):=\sum_{k=1}^{\infty}\frac{1}{(2k-1)^{2N}\Big(\pi^2(2k-1)^2+4t^2\Big)}.
\end{align*}
Obviously, the even function $g(t)>0$ and is strictly decreasing for $t>0$. Hence, for
$t \not=0$, $0<g(t)<g(0)$ and thus $0<\xi(t, N)<1$. The
proof of Theorem \ref{Thm2-remainder-tanh} is complete.
\end{proof}

From \eqref{Thm2-remainder-tanh-expansion},
we obtain the following
\begin{corollary}
For  $t\not=0$, we have
\begin{align}\label{Thm2-remainder-tanh-expansionre}
(-1)^{N}\left(\frac{\tanh t}{t}-\sum_{j=1}^{N}\frac{2^{2j}(2^{2j}-1){B_{2j}}}{(2j)!}t^{2j-2}\right)>0,
\end{align}
that is,
\begin{align}\label{Thm2-Corollary-tanhInequality}
\sum_{j=1}^{2m}\frac{2^{2j}(2^{2j}-1){B_{2j}}}{(2j)!}t^{2j-2}<
\frac{\tanh t}{t}<\sum_{j=1}^{2m-1}\frac{2^{2j}(2^{2j}-1){B_{2j}}}{(2j)!}t^{2j-2}.
\end{align}
\end{corollary}

In analogy with  \eqref{Thm1-tan-Banjac}, we now establish the inequality for $\tanh t/t$.
Write \eqref{Thm2-remainder-tanh-expansion} as
\begin{align}\label{Thm2-remainder-tanh-expansionre}
&(-1)^{N}\left(\frac{\tanh t}{t}-\sum_{j=1}^{N}\frac{2^{2j}(2^{2j}-1){B_{2j}}}{(2j)!}t^{2j-2}\right)\nonumber\\
&\quad=\frac{2^{2N+3}t^{2N}}{\pi^{2N}}\left\{\frac{1}{\pi^2+4t^2}+\sum_{k=2}^{\infty}\frac{1}{(2k-1)^{2N}\Big(\pi^2(2k-1)^2+4t^2\Big)}\right\}.
\end{align}
Noting that the even function
\begin{align*}
G(t):=\sum_{k=2}^{\infty}\frac{1}{(2k-1)^{2N}\Big(\pi^2(2k-1)^2+4t^2\Big)}
\end{align*}
is strictly decreasing for $t>0$, we then obtain from \eqref{Thm2-remainder-tanh-expansionre}  that for $t\not=0$,
\begin{align}\label{Thm2-remainder-tanh-expansionreObtain}
&(-1)^{N}\left(\frac{\tanh t}{t}-\sum_{j=1}^{N}\frac{2^{2j}(2^{2j}-1){B_{2j}}}{(2j)!}t^{2j-2}\right)\nonumber\\
&\qquad\qquad<\frac{2^{2N+3}t^{2N}}{\pi^{2N}}\left\{\frac{1}{\pi^2+4t^2}+\sum_{k=2}^{\infty}\frac{1}{\pi^{2}(2k-1)^{2N+2}}\right\}.
\end{align}
The choice $N=1$ and $N=2$ in \eqref{Thm2-remainder-tanh-expansionreObtain}, respectively, yields
\begin{align}\label{Thm2-remainder-tanh-expansionreObtainN=1}
1-\frac{32t^{2}}{\pi^{2}(\pi^2+4t^2)}-\frac{32t^{2}}{\pi^{4}}\sum_{k=2}^{\infty}\frac{1}{(2k-1)^{4}}<\frac{\tanh t}{t}
\end{align}
and
\begin{align}\label{Thm2-remainder-tanh-expansionreObtainN=2}
\frac{\tanh t}{t}<1-\frac{1}{3}t^{2}+\frac{128t^{4}}{\pi^{4}(\pi^2+4t^2)}+\frac{128t^{4}}{\pi^{6}}\sum_{k=2}^{\infty}\frac{1}{(2k-1)^{6}}.
\end{align}
Noting that
\begin{align*}
\sum_{k=2}^{\infty}\frac{1}{(2k-1)^{4}}=\frac{\pi^4}{96}-1,  \quad \sum_{k=2}^{\infty}\frac{1}{(2k-1)^{6}}=\frac{\pi^6}{960}-1,
\end{align*}
we obtain from \eqref{Thm2-remainder-tanh-expansionreObtainN=1} and \eqref{Thm2-remainder-tanh-expansionreObtainN=2} that for $t\not=0$,
\begin{align}\label{Thm2-remainder-tanh-expansionreObtainN=1N=2}
&\frac{\pi^2+\left(4-\frac{\pi^2}{3}\right)t^2-\left(\frac{4}{3}-\frac{128}{\pi^4}\right)t^4}{\pi^2+4t^2}<\frac{\tanh t}{t}\nonumber\\
&\qquad\qquad\qquad<\frac{\pi^2+\left(4-\frac{\pi^2}{3}\right)t^2-\left(\frac{4}{3}-\frac{2\pi^2}{15}\right)t^4+\left(\frac{8}{15}-\frac{512}{\pi^6}\right)t^6}{\pi^2+4t^2},
\end{align}
which is an analogous result to \eqref{Thm1-tan-Banjac}.

\begin{theorem}\label{Thm-remainder-sec}
Let $N\geq0$ be an integer. Then for $|t|<\pi/2$, we have
\begin{align}\label{Thm1-sec-series-representation}
\sec t=\sum_{j=0}^{N-1}\frac{|E_{2j}|}{(2j)!}t^{2j}+\omega_N(t),
\end{align}
where
\begin{align}\label{Thm1-remainder-sec-series-representation}
\omega_N(t)=\frac{2^{2N+2}t^{2N}}{\pi^{2N-1}}\sum_{k=1}^{\infty}\frac{(-1)^{k+1}}{(2k-1)^{2N-1}\Big(\pi^2(2k-1)^2-4t^2\Big)}.
\end{align}
\end{theorem}

\begin{proof}
It follows from \cite[p. 44]{Gradshteyn2000} that
\begin{equation*}
\sec\frac{\pi x}{2}=\frac{4}{\pi}\sum_{k=1}^{\infty}(-1)^{k+1}\frac{2k-1}{(2k-1)^{2}-x^2}.
\end{equation*}
Replacement of $x$ by $2t/\pi$ yields
\begin{equation}\label{sec-expansion}
\sec t=\frac{4}{\pi}\sum_{k=1}^{\infty}\frac{(-1)^{k+1}}{(2k-1)\Big(1-\big(\frac{2t}{\pi(2k-1)}\big)^2\Big)}.
\end{equation}
Using \eqref{identity1/(1-q)} and the following identity:
\begin{equation}\label{odd-series-Expansion-Euler-constant}
\sum_{k=1}^{\infty}\frac{(-1)^{k+1}}{(2k-1)^{2n+1}}=\frac{\pi^{2n+1}}{2^{2n+1}(2n)!}|E_{2n}|
\end{equation}
(see \cite[p. 8]{Gradshteyn2000}), we obtain from \eqref{sec-expansion}  that
\begin{align*}
\sec t&=\frac{4}{\pi}\sum_{k=1}^{\infty}\frac{(-1)^{k+1}}{2k-1}\left(\sum_{j=0}^{N-1}\left(\frac{2t}{\pi(2k-1)}\right)^{2j}+\frac{\left(\frac{2t}{\pi(2k-1)}\right)^{2N}}{1-\left(\frac{2t}{\pi(2k-1)}\right)^2}\right)\\
&=\sum_{j=0}^{N-1}\frac{|E_{2j}|}{(2j)!}t^{2j}+\omega_N(t),
\end{align*}
where
\begin{align*}
\omega_N(t)=\frac{2^{2N+2}t^{2N}}{\pi^{2N-1}}\sum_{k=1}^{\infty}(-1)^{k+1}\frac{1}{(2k-1)^{2N-1}\Big(\pi^2(2k-1)^2-4t^2\Big)}.
\end{align*}
 The proof of Theorem \ref{Thm-remainder-sec} is complete.
\end{proof}

Chen and Sandor \cite[Theorem 3.1(i)]{Chen-Sandor203--217} proved that for  $0<|t|<\pi/2$,
\begin{equation}\label{sec-inequality1}
\frac{\pi^2}{\pi^2-4t^2}<\sec t<\frac{4\pi}{\pi^2-4t^2}.
\end{equation}
The constants $\pi^2$ and $4\pi$ are  best possible.

Write \eqref{Thm1-sec-series-representation} as
\begin{align}\label{Thm1-sec-series-representationre}
\sec t&=\sum_{j=0}^{N-1}\frac{|E_{2j}|}{(2j)!}t^{2j}+\frac{2^{2N+2}t^{2N}}{\pi^{2N-1}}\left\{\frac{1}{\pi^2-4t^2}+\sum_{k=2}^{\infty}\frac{(-1)^{k+1}}{(2k-1)^{2N-1}\Big(\pi^2(2k-1)^2-4t^2\Big)}\right\}\nonumber\\
&=\sum_{j=0}^{N-1}\frac{|E_{2j}|}{(2j)!}t^{2j}+\frac{2^{2N+2}t^{2N}}{\pi^{2N-1}(\pi^2-4t^2)}\nonumber\\
&\quad+\frac{2^{2N+2}t^{2N}}{\pi^{2N-1}}\sum_{k=2}^{\infty}(-1)^{k+1}\frac{1}{(2k-1)^{2N-1}\Big(\pi^2(2k-1)^2-4t^2\Big)}.
\end{align}

Let
\begin{align*}
H(t)=\sum_{k=2}^{\infty}(-1)^{k+1}\frac{1}{(2k-1)^{2N-1}\Big(\pi^2(2k-1)^2-4t^2\Big)}.
\end{align*}
Differentiation yields
\begin{align*}
H'(t)=-8t\sum_{k=2}^{\infty}(-1)^{k}\eta_{k},\quad  \eta_k=\frac{1}{(2k-1)^{2N-1}\Big(\pi^2(2k-1)^2-4t^2\Big)^2}.
\end{align*}
Then it is easily seen that $\eta_{2k}>\eta_{2k+1}$ for $k\in\mathbb{N}$, $0<t<\pi/2$ and $N\in\mathbb{N}$; thus $H'(t)<0$ for $0<t<\pi/2$.
Hence, for all $0<t<\pi/2$ and $N\in\mathbb{N}$, we have $H(\pi/2)<H(t)<H(0)$. We then obtain from \eqref{Thm1-sec-series-representationre} that for $0<|t|<\frac{\pi}{2}$,
\begin{align}\label{Thm1-sec-series-representationreObtain}
&\sum_{j=0}^{N-1}\frac{|E_{2j}|}{(2j)!}t^{2j}+\frac{2^{2N+2}t^{2N}}{\pi^{2N-1}(\pi^2-4t^2)}+\frac{2^{2N}t^{2N}}{\pi^{2N+1}}\sum_{k=2}^{\infty}\frac{(-1)^{k+1}}{k(k-1)(2k-1)^{2N-1}}\nonumber\\
&\qquad<\sec t<\sum_{j=0}^{N-1}\frac{|E_{2j}|}{(2j)!}t^{2j}+\frac{2^{2N+2}t^{2N}}{\pi^{2N-1}(\pi^2-4t^2)}+\frac{2^{2N+2}t^{2N}}{\pi^{2N+1}}\sum_{k=2}^{\infty}\frac{(-1)^{k+1}}{(2k-1)^{2N+1}}.
\end{align}

Direct computations yield
\begin{align*}
\sum_{k=2}^{\infty}\frac{(-1)^{k+1}}{k(k-1)(2k-1)}=3-\pi,\qquad\qquad \,\,     \sum_{k=2}^{\infty}\frac{(-1)^{k+1}}{(2k-1)^{3}}=\frac{\pi^3}{32}-1.
\end{align*}
The choice $N=1$ in \eqref{Thm1-sec-series-representationreObtain} then yields, for $0<|t|<\frac{\pi}{2}$,
\begin{align}\label{Thm1-sec-series-representationreObtainN=1}
\frac{\pi^2+\frac{28-8\pi}{\pi}t^2+\frac{-48+16\pi}{\pi^3}t^4}{\pi^2-4t^2}<\sec t<\frac{\pi^2-\frac{8-\pi^2}{2}t^2-\frac{4\pi^3-128}{2\pi^3}t^4}{\pi^2-4t^2},
\end{align}
which  improves the inequality \eqref{sec-inequality1}.

\begin{theorem}\label{Thm2-remainder-cot}
For $0<|t|<\pi$, we have
\begin{align}\label{Thm2-cot-series-representation}
\cot t=\frac{1}{t}-\sum_{j=1}^N\frac{2^{2j}{|B_{2j}|}}{(2j)!}t^{2j-1}+\theta_N(t),
\end{align}
where
\begin{align}\label{Thm2-remainder-cot-series-representation}
\theta_N(t)=\frac{2t^{2N+1}}{\pi^{2N}}\sum_{k=1}^{\infty}\frac{1}{k^{2N}\big(t^2-\pi^2k^2\big)}.
\end{align}
\end{theorem}

\begin{proof}
It follows from \cite[p. 118]{Olver-Lozier-Boisvert-Clarks2010} that
\begin{equation}\label{cot-expansion}
\cot t=\frac{1}{t}+2t\sum_{k=1}^{\infty}\frac{1}{t^2-\pi^{2}k^{2}},
\end{equation}
which can be written as
\begin{equation}\label{cot-expansion-re}
\cot t=\frac{1}{t}-2t\sum_{k=1}^{\infty}\frac{1}{(k\pi)^2\left(1-\left(\frac{t}{k\pi}\right)^2\right)}.
\end{equation}

Using \eqref{identity1/(1-q)} and the following identity:
\begin{equation}\label{even-series-Expansion}
\sum_{k=1}^{\infty}\frac{1}{k^{2n}}=\frac{2^{2n-1}\pi^{2n}}{(2n)!}|B_{2n}|
\end{equation}
(see \cite[p. 8]{Gradshteyn2000}), we obtain from \eqref{cot-expansion-re} that
\begin{align*}
\cot t&=\frac{1}{t}-2t\sum_{k=1}^{\infty}\frac{1}{(k\pi)^2}\left(\sum_{j=0}^{N-1}\left(\frac{t}{k\pi}\right)^{2j}+\frac{\left(\frac{t}{k\pi}\right)^{2N}}{1-\left(\frac{t}{k\pi}\right)^2}\right)\\
&=\frac{1}{t}-2\sum_{j=1}^N\frac{2^{2j-1}{|B_{2j}|}}{(2j)!}t^{2j-1}+\theta_N(t),
\end{align*}
where
\begin{align*}
\theta_N(t)=\frac{2t^{2N+1}}{\pi^{2N}}\sum_{k=1}^{\infty}\frac{1}{k^{2N}\big(t^2-\pi^2k^2\big)}.
\end{align*}
 The proof of Theorem \ref{Thm2-remainder-cot} is complete.
\end{proof}

\begin{theorem}\label{Thm2-remainder-1/sin}
For $0<|t|<\pi$, we have
\begin{align}\label{Thm1-csc-series-representation}
\csc t=\frac{1}{t}+\sum_{j=1}^N\frac{(2^{2j}-2){|B_{2j}|}}{(2j)!}t^{2j-1}+r_N(t),
\end{align}
where
\begin{align}\label{Thm1-remainder-csc-series-representation}
r_N(t)=\frac{2t^{2N+1}}{\pi^{2N}}\sum_{k=1}^{\infty}\frac{(-1)^{k+1}}{k^{2N}\big(\pi^2k^2-t^2\big)}.
\end{align}
\end{theorem}

\begin{proof}
It follows from \cite[p. 118]{Olver-Lozier-Boisvert-Clarks2010} that
\begin{equation}\label{csc-expansion}
\csc t=\frac{1}{t}+2t\sum_{k=1}^{\infty}\frac{(-1)^{k+1}}{(k\pi)^{2}-t^2},
\end{equation}
which can be written as
\begin{equation}\label{csc-expansion-re}
\csc t=\frac{1}{t}+2t\sum_{k=1}^{\infty}\frac{(-1)^{k+1}}{(k\pi)^2\left(1-\left(\frac{t}{k\pi}\right)^2\right)}.
\end{equation}

Using \eqref{identity1/(1-q)} and the following identity:
\begin{equation}\label{even-series-ExpansionA}
\sum_{k=1}^{\infty}\frac{(-1)^{k+1}}{k^{2n}}=\frac{(2^{2n-1}-1)\pi^{2n}}{(2n)!}|B_{2n}|
\end{equation}
(see \cite[p. 8]{Gradshteyn2000}), we obtain from \eqref{csc-expansion-re} that
\begin{align*}
\csc t&=\frac{1}{t}+2t\sum_{k=1}^{\infty}\frac{(-1)^{k+1}}{(k\pi)^2}\left(\sum_{j=0}^{N-1}\left(\frac{t}{k\pi}\right)^{2j}+\frac{\left(\frac{t}{k\pi}\right)^{2N}}{1-\left(\frac{t}{k\pi}\right)^2}\right)\\
&=\frac{1}{t}+\sum_{j=1}^N\frac{(2^{2j}-2){|B_{2j}|}}{(2j)!}t^{2j-1}+r_N(t),
\end{align*}
where
\begin{align*}
r_N(t)=\frac{2t^{2N+1}}{\pi^{2N}}\sum_{k=1}^{\infty}\frac{(-1)^{k+1}}{k^{2N}\big(\pi^2k^2-t^2\big)}.
\end{align*}
 The proof of Theorem \ref{Thm2-remainder-1/sin} is complete.
\end{proof}
 Theorems \ref{Thm2-remainder-cot} and \ref{Thm2-remainder-1/sin} are used in Section 3.
\vskip 4mm

\section{Wilker- and Huygens-type inequalities}
Wilker  \cite{WilkerE3306} proposed the following two open problems:
\par
(a) Prove that if $0<x<\pi/2$, then
\begin{equation}\label{Wilker-1}
\left(\frac{\sin x}{x}\right)^{2}+\frac{\tan x}{x}>2.
\end{equation}
\par
(b) Find the largest constant c such that
\begin{equation*}\label{Wilker-2}
\left(\frac{\sin x}{x}\right)^{2}+\frac{\tan x}{x}>2+cx^3\tan x
\end{equation*}
for $0<x<\pi/2$.
\par
In \cite{Sumner264--267}, the inequality \eqref{Wilker-1} was proved,
and the following inequality
\begin{equation}\label{Wilker-II}
2+\left(\frac{2}{\pi}\right)^{4}x^3\tan x<\left(\frac{\sin
x}{x}\right)^{2}+\frac{\tan x}{x}<2+\frac{8}{45}x^3\tan x,\qquad
 0<x<\frac{\pi}{2}
\end{equation}
was also established,  where the constants $(2/\pi)^{4}$ and
$\frac{8}{45}$ are the best possible, .

Chen and Cheung \cite{Chen-CheungJIA} showed  that  for $0<x<\pi/2$,
\begin{equation}\label{Wilker-generalization-ineq-1}
2+\frac{8}{45}x^4+\frac{16}{315}x^5\tan x<\left(\frac{\sin
x}{x}\right)^{2}+\frac{\tan
x}{x}<2+\frac{8}{45}x^4+\left(\frac{2}{\pi}\right)^{6}x^5\tan x,
\end{equation}
where the constants $\frac{16}{315}$ and $(2/\pi)^{6}$
are the best possible, and
\begin{align}\label{Wilker-generalization-ineq-2}
2+\frac{8}{45}x^4+\frac{16}{315}x^6+\frac{104}{4725}x^7\tan
x&<\left(\frac{\sin x}{x}\right)^{2}+\frac{\tan
x}{x}\notag\\
&<2+\frac{8}{45}x^4+\frac{16}{315}x^6+\left(\frac{2}{\pi}\right)^{8}x^7\tan
x,
\end{align}
where the constants $\frac{104}{4725}$ and $(2/\pi)^{8}$ are the best possible.

The Wilker-type inequalities \eqref{Wilker-1} and \eqref{Wilker-II} have
attracted much interest of many mathematicians and have motivated a
large number of research papers involving different proofs, various
generalizations and improvements (cf.
\cite{Baricz397--406,Chen-Cheung,Chen-CheungJIA,Guo19--22,Mortitc535--541,Mortitc516--520,Neuman399--407,Neuman715--723,Neuman271--279,Pinelis905--909,Sumner264--267,
Wu529--535,Wu757--765,Chen-Sandor55--67,Wu683--687,Wu447--458,Zhang149-151,Zhu749--750,ZhuID4858422,Zhu1998--2004,ZhuID130821}
and the references cited therein).

A related inequality that is of interest to us is Huygens'
inequality \cite{Huygens}, which asserts that
\begin{equation}\label{Huygens}
2\left(\frac{\sin x}{x}\right)+\frac{\tan x}{x}>3,\qquad 0<|x|<\frac{\pi}{2}.
\end{equation}
Chen and Cheung \cite{Chen-CheungJIA} showed that  for $0<x<\pi/2$,
\begin{equation}\label{Huygens-generalization-ineq-1}
3+\frac{3}{20}x^3\tan x<2\left(\frac{\sin x}{x}\right)+\frac{\tan
x}{x}<3+\left(\frac{2}{\pi}\right)^{4}x^3\tan x,
\end{equation}
where the constants $\frac{3}{20}$ and $(2/\pi)^{4}$
are the best possible, and
\begin{equation}\label{Huygens-generalization-ineq-2}
3+\frac{3}{20}x^4+\frac{3}{56}x^5\tan x<2\left(\frac{\sin
x}{x}\right)+\frac{\tan
x}{x}<3+\frac{3}{20}x^4+\left(\frac{2}{\pi}\right)^{6}x^5\tan x,
\end{equation}
where the constants $\frac{3}{56}$ and $(2/\pi)^{6}$
are the best possible. These authors also posed three conjectures on Wilker and Huygens-type inequalities. As far as we know, these conjectures have not yet been proved.

 Zhu  \cite{Zhu1180--1182} established some new inequalities of
the Huygens-type for trigonometric and hyperbolic functions. Baricz
and S\'andor \cite{Baricz397--406} pointed out that inequalities
\eqref{Wilker-1} and \eqref{Huygens} are simple consequences of the
arithmetic-geometric mean inequality, together with the well-known
Lazarevi\'c-type inequality \cite[p. 238]{Mitrinovic238}
\begin{equation*}
(\cos x)^{1/3}<\frac{\sin x}{x},\qquad
0<|x|<\frac{\pi}{2},
\end{equation*}
or equivalently
\begin{equation}\label{Lazarevic-type-inequality}
\left(\frac{\sin x}{x}\right)^{2}\frac{\tan x}{x}>1,\qquad 0<|x|<\frac{\pi}{2}.
\end{equation}

Wu and Srivastava \cite[Lemma 3]{Wu529--535} established another
inequality
\begin{equation}\label{Wilker-2}
\left(\frac{x}{\sin x}\right)^{2}+\frac{x}{\tan x}>2, \qquad 0<|x|<\frac{\pi}{2}.
\end{equation}
Neuman and S\'andor \cite[Theorem
2.3]{Neuman715--723} proved that for $0<|x|<\pi/2$,
\begin{align}\label{NeumanSandorThm2.3}
\frac{\sin x}{x}<\frac{2+\cos x}{3}<\frac{1}{2}\left(\frac{x}{\sin
x}+\cos x\right).
\end{align}
By multiplying both sides of  inequality \eqref{NeumanSandorThm2.3}
by $x/\sin x$, we obtain that for $0<|x|<\pi/2$,
\begin{align}\label{Cusa-re}
\frac{1}{2}\left[\left(\frac{x}{\sin x}\right)^{2}+\frac{x}{\tan
x}\right]>\frac{2(x/\sin x)+x/\tan x}{3}>1.
\end{align}

 Chen and S\'andor \cite{Chen-Sandor55--67} established the following inequality
 chain:
\begin{align}\label{LWsuggest}
&\quad\frac{\left(\sin x/x\right)^{2}+\tan x/x}{2}>\left(\frac{\sin
x}{x}\right)^{2}\left(\frac{\tan x}{x}\right)>\frac{2\left(\sin
x/x\right)+\tan x/x}{3}\notag\\
&> \left(\frac{\sin x}{x}\right)^{2/3}\left(\frac{\tan
x}{x}\right)^{1/3}>\frac{1}{2}\left[\left(\frac{x}{\sin
x}\right)^{2}+\frac{x}{\tan x}\right]>\frac{2(x/\sin x)+x/\tan
x}{3}>1
\end{align}
for $0<|x|<\pi/2$.

In this section, we  develop the inequality \eqref{Wilker-2} and the last inequality in \eqref{LWsuggest} to produce   sharp inequalities.

It is well known \cite[p. 42]{Gradshteyn2000} that
\begin{align}\label{expansion-cotx}
\cot x=\frac{1}{x}-\sum_{k=1}^{\infty}\frac{2^{2k}|B_{2k}|}{(2k)!}x^{2k-1},\qquad |x|<\pi.
\end{align}
Differentiating  the expression in \eqref{expansion-cotx}, we find
\begin{align}\label{expansion-csc2x}
\left(\frac{x}{\sin x}\right)^2=1+\sum_{k=1}^{\infty}\frac{2^{2k}(2k-1)|B_{2k}|}{(2k)!}x^{2k},\qquad |x|<\pi.
\end{align}
From \eqref{expansion-cotx} and \eqref{expansion-csc2x}, we obtain that for $|x|<\pi$,
\begin{align}\label{expansion-Wilker}
\left(\frac{x}{\sin x}\right)^2+\frac{x}{\tan x}=2+\sum_{k=1}^{\infty}\frac{k\cdot2^{2k+3}|B_{2k+2}|}{(2k+2)!}x^{2k+2}.
\end{align}
It follows from \eqref{expansion-Wilker} that for every $N\in\mathbb{N}$,
\begin{align}\label{expansion-Wilkerholds}
\frac{N\cdot2^{2N+3}|B_{2N+2}|}{(2N+2)!}x^{2N+2}<\left(\frac{x}{\sin x}\right)^2+\frac{x}{\tan x}-\left(2+\sum_{k=1}^{N-1}\frac{k\cdot2^{2k+3}|B_{2k+2}|}{(2k+2)!}x^{2k+2}\right)
\end{align}
for $0<|x|<\pi$.

In view of \eqref{expansion-Wilkerholds} it is natural to ask: What is the largest number $\lambda_{N}$ and what is the smallest number $\mu_{N}$ such that the inequality
\begin{align*}
\lambda_{N}x^{2N+2}<\left(\frac{x}{\sin x}\right)^2+\frac{x}{\tan x}-\left(2+\sum_{k=1}^{N-1}\frac{k\cdot2^{2k+3}|B_{2k+2}|}{(2k+2)!}x^{2k+2}\right)<\mu_{N}x^{2N+2}
\end{align*}
holds for  $x\in(0, \pi/2)$ and $N\in\mathbb{N}$? Theorem \ref{Wilker-inequality-lambda-mu} answers this question.

\begin{theorem}\label{Wilker-inequality-lambda-mu}
Let $N\geq1$ be an integer. Then for $0<t<\pi/2$,
\begin{equation}\label{Sharp-Wilkerlambda-mu}
\lambda_{N}t^{2N+2}<\left(\frac{t}{\sin t}\right)^2+\frac{t}{\tan t}-\left(2+\sum_{j=1}^{N-1}\frac{j\cdot2^{2j+3}|B_{2j+2}|}{(2j+2)!}t^{2j+2}\right)<\mu_{N}t^{2N+2}
\end{equation}
with the best possible constants
\begin{equation}\label{Wilker-inequality-lambda-Best-lambda}
\lambda_{N}=\frac{N\cdot2^{2N+3}|B_{2N+2}|}{(2N+2)!}
\end{equation}
and
\begin{align}\label{Wilker-inequality-lambda-Best-mu}
\mu_{N}=\frac{64N}{\pi^{2N+2}}\sum_{k=1}^{\infty}\frac{1}{k^{2N-2}(4k^2-1)^2}-\frac{16(N-1)}{\pi^{2N+2}}\sum_{k=1}^{\infty}\frac{1}{k^{2N}(4k^2-1)^2}.
\end{align}
\end{theorem}

\begin{proof}
It follows from \cite[p. 44]{Gradshteyn2000} that
\begin{equation}\label{csc2xsum}
\csc^2(\pi x)=\frac{1}{\pi^2x^2}+\frac{2}{\pi^2}\sum_{k=1}^{\infty}\frac{x^2+k^2}{(x^2-k^2)^2}.
\end{equation}
Replacement of $x$ by $t/\pi$ in \eqref{csc2xsum} yields
\begin{equation}\label{csc2tsum}
\left(\frac{t}{\sin t}\right)^2=1+\sum_{k=1}^{\infty}\frac{2t^4+2\pi^{2}k^{2}t^2}{(t^2-\pi^{2}k^2)^2}.
\end{equation}
From \eqref{csc2tsum} and \eqref{cot-expansion}, we obtain
\begin{align}\label{Wilker-inequality-Keyresult}
\left(\frac{t}{\sin t}\right)^2+\frac{t}{\tan t}=2+4t^4\sum_{k=1}^{\infty}\frac{1}{(\pi^2k^2-t^2)^2}=2+4t^4\sum_{k=1}^{\infty}\frac{1}{\pi^{4}k^{4}\big(1-(\frac{t}{\pi k})^2\big)^2}.
\end{align}
Using the following identity:
\begin{equation}\label{identity1/(1-q)2}
\frac{1}{(1-q)^2}=\sum_{j=1}^{N-1}jq^{j-1}+\frac{Nq^{N-1}}{1-q}+\frac{q^N}{(1-q)^2}\qquad
(q\not=1)
\end{equation}
and \eqref{even-series-Expansion}, we then have
\begin{align*}
&\left(\frac{t}{\sin t}\right)^2+\frac{t}{\tan t}=2+4t^4\sum_{k=1}^{\infty}\frac{1}{\pi^{4}k^{4}\big(1-(\frac{t}{\pi k})^2\big)^2}\\
&\quad=2+4t^4\sum_{k=1}^{\infty}\frac{1}{\pi^{4}k^{4}}\left(\sum_{j=1}^{N-1}j\left(\frac{t}{\pi k}\right)^{2j-2}+\frac{N(\frac{t}{\pi k})^{2N-2}}{1-(\frac{t}{\pi k})^2}+\frac{(\frac{t}{\pi k})^{2N}}{(1-(\frac{t}{\pi k})^2)^2}\right)\\
&\quad=2+\sum_{j=1}^{N-1}\frac{j\cdot2^{2j+3}|B_{2j+2}|}{(2j+2)!}t^{2j+2}+\sum_{k=1}^{\infty}\frac{4Nt^{2N+2}}{\pi^{2N}k^{2N}(\pi^{2}k^2-t^2)}+\sum_{k=1}^{\infty}\frac{4t^{2N+4}}{\pi^{2N}k^{2N}(\pi^{2}k^2-t^2)^2}\\
&\quad=2+\sum_{j=1}^{N-1}\frac{j\cdot2^{2j+3}|B_{2j+2}|}{(2j+2)!}t^{2j+2}+\frac{4t^{2N+2}}{\pi^{2N}}V_{N}(t),
\end{align*}
where
\begin{align*}
V_{N}(t)=\sum_{k=1}^{\infty}\frac{N\pi^2k^2-(N-1)t^2}{k^{2N}(\pi^{2}k^2-t^2)^2}.
\end{align*}

Differentiation yields
\begin{align*}
V'_{N}(t)=\sum_{k=1}^{\infty}\frac{2t\Big((N+1)\pi^2k^2-(N-1)t^2\Big)}{k^{2N}(\pi^{2}k^2-t^2)^3}>0.
\end{align*}
Hence, $V_{N}(t)$ is strictly increasing for $t\in (0, \pi/2)$, and we have
\begin{equation*}\label{Sharp-Wilkerlambda-mu}
\lambda_{N}t^{2N+2}<\left(\frac{t}{\sin t}\right)^2+\frac{t}{\tan t}-\left(2+\sum_{j=1}^{N-1}\frac{j\cdot2^{2j+3}|B_{2j+2}|}{(2j+2)!}t^{2j+2}\right)<\mu_{N}t^{2N+2}
\end{equation*}
with
\begin{align*}
\lambda_{N}=\frac{4}{\pi^{2N}}V_{N}(0)\quad\text{and}\quad \mu_{N}&=\frac{4}{\pi^{2N}}V_{N}\left(\frac{\pi}{2}\right).
\end{align*}
Direct computations yield
\begin{align*}
V_{N}(0)=\frac{N}{\pi^2}\sum_{k=1}^{\infty}\frac{1}{k^{2N+2}}=\frac{N\cdot2^{2N+1}\pi^{2N}|B_{2N+2}|}{(2N+2)!}
\end{align*}
and
\begin{align*}
V_{N}\left(\frac{\pi}{2}\right)&=\frac{16N}{\pi^2}\sum_{k=1}^{\infty}\frac{1}{k^{2N-2}(4k^2-1)^2}-\frac{4(N-1)}{\pi^2}\sum_{k=1}^{\infty}\frac{1}{k^{2N}(4k^2-1)^2}.
\end{align*}
Hence, the inequality \eqref{Sharp-Wilkerlambda-mu} holds with the best possible constants given in \eqref{Wilker-inequality-lambda-Best-lambda} and \eqref{Wilker-inequality-lambda-Best-mu}.
The proof of Theorem \ref{Wilker-inequality-lambda-mu} is complete.
\end{proof}

\begin{remark}
Direct computations yield
\begin{equation*}
\lambda_{1}=\frac{2}{45},\quad \mu_{1}=\frac{4(\pi^2-8)}{\pi^4}
\end{equation*}
and
\begin{equation*}
\lambda_2=\frac{8}{945},\quad \mu_2=\frac{8(-720+90\pi^2-\pi^4)}{45\pi^6}.
\end{equation*}
We then obtain from \eqref{Sharp-Wilkerlambda-mu} that for $0<t<\pi/2$,
\begin{equation}\label{Sharp-Wilkerlambda-mu-N=1}
2+\frac{2}{45}t^{4}<\left(\frac{t}{\sin t}\right)^2+\frac{t}{\tan t}<2+\frac{4(\pi^2-8)}{\pi^4}t^{4},
\end{equation}
where the constants $\frac{2}{45}$ and $4(\pi^2-8)/\pi^4$
are the best possible, and
\begin{equation}\label{Sharp-Wilkerlambda-mu-N=2}
2+\frac{2}{45}t^{4}+\frac{8}{945}t^6<\left(\frac{t}{\sin t}\right)^2+\frac{t}{\tan t}<2+\frac{2}{45}t^{4}+\frac{8(-720+90\pi^2-\pi^4)}{45\pi^6}t^{6},
\end{equation}
where the constants $\frac{8}{945}$ and $8(-720+90\pi^2-\pi^4)/(45\pi^6)$
are the best possible.
\end{remark}

The classical Euler gamma function may be defined (for $x>0$) by
\begin{equation}\label{egamma}
\Gamma(x)=\int^\infty_0t^{x-1} e^{-t}{\rm d} t.
\end{equation}
The logarithmic derivative of $\Gamma(x)$, denoted by $$\psi(x)=\frac{\Gamma'(x)}{\Gamma(x)},$$ is called the psi
(or digamma) function, and $\psi^{(k)}(x)\;\; (k\in \mathbb{N})$ are called the polygamma functions.

\begin{theorem}\label{Sharp-Wilker-type-inequality2}
Let $N\geq0$ be an integer. Then for $0<x<\pi/2$,
\begin{align}\label{Best-Wilker-type-inequality-x}
\alpha_{N}x^{4}<\left(\frac{x}{\sin x}\right)^2+\frac{x}{\tan x}-\left(2+4x^4\sum_{k=1}^{N}\frac{1}{(\pi^2k^2-x^2)^2}\right)<\beta_{N}x^{4}
\end{align}
with the best possible constants
\begin{align}\label{Huygens-inequality-best-constants-alphaN}
\alpha_N=\frac{2\psi'''(N+1)}{3\pi^4}\quad \text{and}\quad
\beta_N&=\frac{8\Big((2N+1)^2\psi'(N+\frac{1}{2})-4(N+1)\Big)}{(2N+1)^2\pi^4}.
\end{align}
\end{theorem}

\begin{proof}
Write \eqref{Wilker-inequality-Keyresult} as
\begin{align*}
\left(\frac{x}{\sin x}\right)^2+\frac{x}{\tan x}=2+4x^4\sum_{k=1}^{N}\frac{1}{(\pi^2k^2-x^2)^2}+4x^4A_N(x),
\end{align*}
where
\begin{align*}
A_N(x)=\sum_{k=N+1}^{\infty}\frac{1}{(\pi^2k^2-x^2)^2}.
\end{align*}
Obviously, $A_N(x)$ is strictly increasing for $x\in(0, \pi/2)$. Hence, for $0<x<\pi/2$, we have
\begin{align*}
\alpha_Nx^4<\left(\frac{x}{\sin x}\right)^2+\frac{x}{\tan x}-\left(2+4x^4\sum_{k=1}^{N}\frac{1}{(\pi^2k^2-x^2)^2}\right)<\beta_{N}x^4
\end{align*}
with
\begin{align*}
\alpha_N=4A_N(0)=\frac{4}{\pi^4}\sum_{k=N+1}^{\infty}\frac{1}{k^4}  \quad\text{and}\quad    \beta_N=4A_N\left(\frac{\pi}{2}\right)=\frac{64}{\pi^4}\sum_{k=N+1}^{\infty}\frac{1}{\big(4k^2-1\big)^2}.
\end{align*}

From the following formula (see \cite[p.~260, Eq. (6.4.10)]{abram}):
\begin{align*}
\psi^{(n)}(z)=(-1)^{n+1}n!\sum_{k=0}^{\infty}\frac{1}{(z+k)^{n+1}},\qquad  z\not= 0,-1,-2,\ldots,
\end{align*}
we obtain
\begin{align}\label{Sharp-Wilker-type-inequality-find1}
\sum_{k=N+1}^{\infty}\frac{1}{k^4}=\frac{\psi'''(N+1)}{6}.
\end{align}
We find\footnote{The formula \eqref{Sharp-Wilker-type-inequality-find2} is established by induction on $N$ in the appendix.}
\begin{align}\label{Sharp-Wilker-type-inequality-find2}
\sum_{k=N+1}^{\infty}\frac{1}{\big(4k^2-1\big)^2}=\frac{1}{8}\psi'\left(N+\frac{1}{2}\right)-\frac{N+1}{2(2N+1)^2}.
\end{align}
Hence, the inequality \eqref{Best-Wilker-type-inequality-x} holds with the best possible constants given in \eqref{Huygens-inequality-best-constants-alphaN}.
The proof of Theorem \ref{Sharp-Wilker-type-inequality2} is complete.
\end{proof}

\begin{remark}\label{Sharp-Wilker-type-inequality-xRemark1}
The choice $N=0$ in \eqref{Best-Wilker-type-inequality-x} yields \eqref{Sharp-Wilkerlambda-mu-N=1}. The choice $N=1$ in \eqref{Best-Wilker-type-inequality-x} yields
\begin{align}\label{Best-Wilker-type-inequality-N=1}
2+\frac{4x^4}{(\pi^2-x^2)^2}+\frac{2(\pi^4-90)}{45\pi^4}x^{4}<\left(\frac{x}{\sin x}\right)^2+\frac{x}{\tan x}<2+\frac{4x^4}{(\pi^2-x^2)^2}+\frac{4(9\pi^2-88)}{9\pi^4}x^{4}
\end{align}
for $0<x<\pi/2$, where the constants $2(\pi^4-90)/(45\pi^4)$ and $4(9\pi^2-88)/(9\pi^4)$
are the best possible.
\end{remark}
\begin{remark}
There is no strict comparison between  the two lower bounds in \eqref{Sharp-Wilkerlambda-mu-N=2} and \eqref{Best-Wilker-type-inequality-N=1}. Likewise, there is no strict comparison between  the two upper bounds in \eqref{Sharp-Wilkerlambda-mu-N=2} and \eqref{Best-Wilker-type-inequality-N=1}.
\end{remark}

 Theorem \ref{Walker-type-Inequality-Proof1} proves Conjecture 2 in \cite{Chen-CheungJIA}.
\begin{theorem}\label{Walker-type-Inequality-Proof1}
Let $N\geq1$ be an integer. Then for $0<x<\pi/2$, we have
\begin{align}\label{expansion-Wilker-Conjecture}
&2+\sum_{k=1}^{N-1}\frac{k\cdot2^{2k+3}|B_{2k+2}|}{(2k+2)!}x^{2k+2}+p_{N}x^{2N+1}\tan x<\left(\frac{x}{\sin x}\right)^2+\frac{x}{\tan x}\nonumber\\
&\qquad\qquad\qquad\qquad\qquad<2+\sum_{k=1}^{N-1}\frac{k\cdot2^{2k+3}|B_{2k+2}|}{(2k+2)!}x^{2k+2}+q_{N}x^{2N+1}\tan x
\end{align}
with the best possible constants
\begin{align}\label{expansion-Wilker-ConjectureConstants}
p_N=0\quad\text{and}\quad q_N=\frac{N\cdot2^{2N+3}|B_{2N+2}|}{(2N+2)!}.
\end{align}
\end{theorem}

\begin{proof}
By \eqref{expansion-Wilker}, for $p_N=0$, the first inequality in \eqref{expansion-Wilker-Conjecture} holds. We now prove the second inequality in \eqref{expansion-Wilker-Conjecture} with $q_N=N\cdot2^{2N+3}|B_{2N+2}|/(2N+2)!$.
Using \eqref{expansion-Wilker} and the following expansion (see \cite[p. 42]{Gradshteyn2000}):
\begin{align}\label{tanx-expansion}
\tan x=\sum_{k=1}^{\infty}\frac{2^{2k}(2^{2k}-1)|B_{2k}|}{(2k)!}x^{2k-1},\qquad |x|<\frac{\pi}{2},
\end{align}
we  find
\begin{align}\label{Thm-Wilker-type-inequality2Key}
&\frac{N\cdot2^{2N+3}|B_{2N+2}|}{(2N+2)!}x^{2N+1}\tan x-\left(\left(\frac{x}{\sin x}\right)^2+\frac{x}{\tan x}-2-\sum_{k=1}^{N-1}\frac{k\cdot2^{2k+3}|B_{2k+2}|}{(2k+2)!}x^{2k+2}\right)\nonumber\\
&=\frac{N\cdot2^{2N+3}|B_{2N+2}|}{(2N+2)!}x^{2N+1}\sum_{k=1}^{\infty}\frac{2^{2k}(2^{2k}-1)|B_{2k}|}{(2k)!}x^{2k-1}-\sum_{k=N+1}^{\infty}\frac{(k-1)\cdot2^{2k+1}|B_{2k}|}{(2k)!}x^{2k}\nonumber\\
&=\sum_{k=N+2}^{\infty}\left\{\frac{N\cdot2^{2N+3}|B_{2N+2}|}{(2N+2)!}\frac{2^{2k-2N}(2^{2k-2N}-1)|B_{2k-2N}|}{(2k-2N)!}-\frac{(k-1)\cdot2^{2k+1}|B_{2k}|}{(2k)!}\right\}x^{2k},
\end{align}
where we note that the term corresponding to $k=N+1$ vanishes.

We claim that  for $k\geq N+2$,
\begin{align}\label{Thm-Huygens-type-inequality2Keyclaim}
\frac{N\cdot2^{2N+3}|B_{2N+2}|}{(2N+2)!}\frac{2^{2k-2N}(2^{2k-2N}-1)|B_{2k-2N}|}{(2k-2N)!}>\frac{(k-1)\cdot2^{2k+1}|B_{2k}|}{(2k)!}.
\end{align}
Using the inequality (see \cite[p. 805]{abram})
 \begin{eqnarray}\label{eq:1.11}
 \frac{2}{\left(2\pi\right)^{2n}\left(1-2^{1-2n}\right)}>\frac{\left|B_{2n}\right|}{(2n)!}>\frac{2}{\left(2\pi\right)^{2n}},\qquad
n\geq1,
\end{eqnarray}
 it is sufficient to prove that for $k\geq N+2$,
 \begin{align*}
\frac{N\cdot2^{2N+3}\cdot2}{\left(2\pi\right)^{2N+2}}\frac{2^{2k-2N}(2^{2k-2N}-1)\cdot2}{\left(2\pi\right)^{2k-2N}}> \frac{2(k-1)\cdot2^{2k+1}}{\left(2\pi\right)^{2k}\left(1-2^{1-2k}\right)},
\end{align*}
which can be rearranged as
\begin{align*}
N\left(\frac{2^{2k}}{2^{2N}}-1\right)>\frac{\pi^2}{2}(k-1)\left(1+\frac{2}{2^{2k}-2}\right),\qquad k\geq N+2.
\end{align*}
Noting that $\pi^2/2<5$, it is enough to prove the following inequality:
\begin{align*}
N\left(\frac{2^{2k}}{2^{2N}}-1\right)>5(k-1)\left(1+\frac{2}{2^{2k}-2}\right),\qquad k\geq N+2,
\end{align*}
which can be rearranged as
\begin{align*}
\frac{N}{2^{2N}}2^{2k}-5(k-1)>N+\frac{10(k-1)}{2^{2k}-2},\qquad k\geq N+2.
\end{align*}
Noting that the sequence
\begin{align*}
\frac{N}{2^{2N}}2^{2k}-5(k-1)
\end{align*}
is strictly increasing for $k\geq N+2$, and  the sequence
\begin{align*}
\frac{10(k-1)}{2^{2k}-2}
\end{align*}
is strictly decreasing for $k\geq 2$, it is enough to prove the following inequality:
\begin{align*}
\frac{N}{2^{2N}}2^{2(N+2)}-5(N+1)>N+\frac{10(N+1)}{2^{2(N+2)}-2},
\end{align*}
which can be rearranged as
\begin{align*}
(2N-1)2^{2N+3}>3N,\qquad N\geq1.
\end{align*}
Obviously, the last inequality holds.  This proves the claim \eqref{Thm-Huygens-type-inequality2Keyclaim}.
From \eqref{Thm-Wilker-type-inequality2Key}, we obtain the second inequality in \eqref{expansion-Wilker-Conjecture} with $q_N=N\cdot2^{2N+3}|B_{2N+2}|/(2N+2)!$.

Write \eqref{expansion-Wilker-Conjecture} as
\begin{align*}
p_N<\frac{\left(\frac{x}{\sin x}\right)^2+\frac{x}{\tan x}-2-\sum_{k=1}^{N-1}\frac{k\cdot2^{2k+3}|B_{2k+2}|}{(2k+2)!}x^{2k+2}}{x^{2N+1}\tan x}<q_N.
\end{align*}
We find that
\begin{align*}
\lim_{x\to\frac{\pi}{2}}\frac{\left(\frac{x}{\sin x}\right)^2+\frac{x}{\tan x}-2-\sum_{k=1}^{N-1}\frac{k\cdot2^{2k+3}|B_{2k+2}|}{(2k+2)!}x^{2k+2}}{x^{2N+1}\tan x}=0
\end{align*}
and
\begin{align*}
\lim_{x\to0}\frac{\left(\frac{x}{\sin x}\right)^2+\frac{x}{\tan x}-2-\sum_{k=1}^{N-1}\frac{k\cdot2^{2k+3}|B_{2k+2}|}{(2k+2)!}x^{2k+2}}{x^{2N+1}\tan x}=\frac{N\cdot2^{2N+3}|B_{2N+2}|}{(2N+2)!}.
\end{align*}
Hence, the inequality \eqref{expansion-Wilker-Conjecture} holds with the best possible constants given in \eqref{expansion-Wilker-ConjectureConstants}. The proof of Theorem \ref{Walker-type-Inequality-Proof1} is complete.
\end{proof}

Using \eqref{expansion-cotx} and the following expansion  (see \cite[p. 43]{Gradshteyn2000}):
\begin{align*}
\csc x=\frac{1}{x}+\sum_{k=1}^{\infty}\frac{2(2^{2k-1}-1)|B_{2k}|}{(2k)!}x^{2k-1},\qquad |x|<\pi,
\end{align*}
we  find
\begin{align}\label{Thm-Huygens-type-inequality2Key}
2\left(\frac{x}{\sin x}\right)+\frac{x}{\tan x}=3+\sum_{k=2}^{\infty}\frac{(2^{2k}-4)|B_{2k}|}{(2k)!}x^{2k},\qquad |x|<\pi.
\end{align}
It follows from \eqref{Thm-Huygens-type-inequality2Key} that for every $N\in\mathbb{N}$,
\begin{align}\label{expansion-Huygensholds}
\frac{(2^{2N+2}-4)|B_{2N+2}|}{(2N+2)!}x^{2N+2}<2\left(\frac{x}{\sin x}\right)+\frac{x}{\tan x}-\left(3+\sum_{k=2}^{N}\frac{(2^{2k}-4)|B_{2k}|}{(2k)!}x^{2k}\right)
\end{align}
for $0<|x|<\pi$.

In view of \eqref{expansion-Huygensholds} it is natural to ask: What is the largest number $a_{N}$ and what is the smallest number $b_{N}$ such that the inequality
\begin{align*}
a_{N}x^{2N+2}<2\left(\frac{x}{\sin x}\right)+\frac{x}{\tan x}-\left(3+\sum_{k=2}^{N}\frac{(2^{2k}-4)|B_{2k}|}{(2k)!}x^{2k}\right)<b_{N}x^{2N+2}
\end{align*}
holds for  $x\in(0, \pi/2)$ and $N\in\mathbb{N}$? Theorem \ref{Sharp-Huygens-type-inequality-x} answers this question.

\begin{theorem}\label{Sharp-Huygens-type-inequality-x}
Let $N\geq1$ be an integer. Then for $0<|x|<\pi/2$,
\begin{align}\label{Best-Huygens-type-inequality-x}
a_{N}x^{2N+2}<2\left(\frac{x}{\sin x}\right)+\frac{x}{\tan x}-\left(3+\sum_{j=2}^N\frac{(2^{2j}-4){|B_{2j}|}}{(2j)!}x^{2j}\right)<b_{N}x^{2N+2}
\end{align}
with the best possible constants
\begin{align}\label{Huygens-inequality-best-constants-aN}
a_N=\frac{(2^{2N+2}-4)|B_{2N+2}|}{(2N+2)!}
\end{align}
and
\begin{align}\label{Huygens-inequality-best-constants-bN}
b_N&=\frac{8}{\pi^{2N+2}}\left(\sum_{k=1}^{\infty}\frac{(-1)^{k+1}}{k^{2N}(2k-1)}-\sum_{k=1}^{\infty}\frac{(-1)^{k+1}}{k^{2N}(2k+1)}\right)\nonumber\\
&\quad-\frac{4}{\pi^{2N+2}}\left(\sum_{k=1}^{\infty}\frac{1}{k^{2N}(2k-1)}-\sum_{k=1}^{\infty}\frac{1}{k^{2N}(2k+1)}\right).
\end{align}
\end{theorem}

\begin{proof}
By Theorems \ref{Thm2-remainder-1/sin} and \ref{Thm2-remainder-cot}, we have
\begin{align}\label{remainder-1/sin}
2\left(\frac{x}{\sin x}\right)=2+\sum_{j=1}^N\frac{(2^{2j+1}-4){|B_{2j}|}}{(2j)!}x^{2j}+x^{2N+2}\sum_{k=1}^{\infty}\frac{4(-1)^{k+1}}{(k\pi)^{2N}\big((k\pi)^2-x^2\big)}
\end{align}
and
\begin{align}\label{remainder-tan}
\frac{x}{\tan x}=1-\sum_{j=1}^N\frac{2^{2j}{|B_{2j}|}}{(2j)!}x^{2j}-x^{2N+2}\sum_{k=1}^{\infty}\frac{2}{(k\pi)^{2N}\big((k\pi)^2-x^2\big)}.
\end{align}
Adding these two expressions, we obtain
\begin{align}\label{Huygens-inequality-generalization}
2\left(\frac{x}{\sin x}\right)+\frac{x}{\tan x}=3+\sum_{j=2}^N\frac{(2^{2j}-4){|B_{2j}|}}{(2j)!}x^{2j}+\frac{2x^{2N+2}}{\pi^{2N}}U_{N}(x),
\end{align}
where
\begin{align*}
U_{N}(x)=\sum_{k=1}^{\infty}(-1)^{k+1}\frac{2-(-1)^{k+1}}{k^{2N}\big((k\pi)^2-x^2\big)}.
\end{align*}
Differentiation yields
\begin{align}\label{Huygens-inequality-diff-Ux}
\frac{U'_{N}(x)}{2x}=\sum_{k=1}^{\infty}(-1)^{k+1}\alpha_k,\quad \alpha_k=\frac{2-(-1)^{k+1}}{k^{2N}\big((k\pi)^2-x^2\big)^2}.
\end{align}
Then it is easily seen that $\alpha_{k}>\alpha_{k+1}$ for $k\in\mathbb{N}$, $0<x<\pi/2$ and $N\in\mathbb{N}$; thus for every $N\geq1$, we have $U'_N(x)>0$ for $0<x<\pi/2$.
Hence, for all $0<x<\pi/2$ and $N\in\mathbb{N}$, we have
\begin{align*}
U_{N}(0)<U_{N}(x)<U_{N}\left(\frac{\pi}{2}\right).
\end{align*}

Using \eqref{even-series-Expansion} and \eqref{even-series-ExpansionA}, we find
\begin{align*}
a_N&=\frac{2U_{N}(0)}{\pi^{2N}}=4\sum_{k=1}^{\infty}\frac{(-1)^{k+1}}{(k\pi)^{2N+2}}-2\sum_{k=1}^{\infty}\frac{1}{(k\pi)^{2N+2}}\\
&=\frac{4(2^{2N+1}-1)}{(2N+2)!}|B_{2N+2}|-\frac{2\cdot2^{2N+1}}{(2N+2)!}|B_{2N+2}|=\frac{(2^{2N+2}-4)|B_{2N+2}|}{(2N+2)!}
\end{align*}
and
\begin{align*}
b_N&=\frac{2U_{N}(\pi/2)}{\pi^{2N}}=4\sum_{k=1}^{\infty}\frac{(-1)^{k+1}}{(k\pi)^{2N}\big((k\pi)^2-(\pi/2)^2\big)}-2\sum_{k=1}^{\infty}\frac{1}{(k\pi)^{2N}\big((k\pi)^2-(\pi/2)^2\big)}\\
&=\frac{8}{\pi^{2N+2}}\left(\sum_{k=1}^{\infty}\frac{(-1)^{k+1}}{k^{2N}(2k-1)}-\sum_{k=1}^{\infty}\frac{(-1)^{k+1}}{k^{2N}(2k+1)}\right)\\
&\quad-\frac{4}{\pi^{2N+2}}\left(\sum_{k=1}^{\infty}\frac{1}{k^{2N}(2k-1)}-\sum_{k=1}^{\infty}\frac{1}{k^{2N}(2k+1)}\right).
\end{align*}
The proof of Theorem \ref{Sharp-Huygens-type-inequality-x} is complete.
\end{proof}

Clearly,
\begin{align*}
a_1=\frac{1}{60}\quad \text{and}\quad a_2=\frac{1}{504}.
\end{align*}
Direct computations yield
\begin{align*}
\sum_{k=1}^{\infty}\frac{(-1)^{k+1}}{k^{2}(2k-1)}=\pi-2\ln2-\frac{\pi^2}{12},\quad \sum_{k=1}^{\infty}\frac{(-1)^{k+1}}{k^{2}(2k+1)}=4-2\ln2-\pi+\frac{\pi^2}{12},
\end{align*}
\begin{align*}
\sum_{k=1}^{\infty}\frac{1}{k^{2}(2k-1)}=-\frac{\pi^2}{6}+4\ln2,\quad \sum_{k=1}^{\infty}\frac{1}{k^{2}(2k+1)}=-4+4\ln2+\frac{\pi^2}{6},
\end{align*}
\begin{align*}
\sum_{k=1}^{\infty}\frac{(-1)^{k+1}}{k^{4}(2k-1)}=4\pi-8\ln2-\frac{\pi^2}{3}-\frac{3}{2}\zeta(3)-\frac{7\pi^4}{720},
\end{align*}
\begin{align*}
 \sum_{k=1}^{\infty}\frac{(-1)^{k+1}}{k^{4}(2k+1)}=16-4\pi-8\ln2+\frac{\pi^2}{3}-\frac{3}{2}\zeta(3)+\frac{7\pi^4}{720},
\end{align*}
\begin{align*}
\sum_{k=1}^{\infty}\frac{1}{k^{4}(2k-1)}=16\ln2-\frac{2\pi^2}{3}-2\zeta(3)-\frac{\pi^4}{90},
\end{align*}
\begin{align*}
\sum_{k=1}^{\infty}\frac{1}{k^{4}(2k+1)}=-16+16\ln2+\frac{2\pi^2}{3}-2\zeta(3)+\frac{\pi^4}{90},
\end{align*}
where $\zeta(s)$ is the Riemann zeta function.
Then, we obtain from \eqref{Huygens-inequality-best-constants-bN}
\begin{align*}
b_1=\frac{16(\pi-3)}{\pi^4}  \quad\text{and}\quad b_2=\frac{960\pi-\pi^4-2880}{15\pi^6}.
\end{align*}

From \eqref{Best-Huygens-type-inequality-x}, we have, for $0<|x|<\pi/2$,
\begin{align}\label{Best-Huygens-type-inequality-xN=1}
3+\frac{1}{60}x^{4}<2\left(\frac{x}{\sin x}\right)+\frac{x}{\tan x}<3+\frac{16(\pi-3)}{\pi^4}x^{4},
\end{align}
where the constants $\frac{1}{60}$ and $16(\pi-3)/\pi^4$
are the best possible, and
\begin{align}\label{Best-Huygens-type-inequality-xN=2}
3+\frac{1}{60}x^4+\frac{1}{504}x^{6}<2\left(\frac{x}{\sin x}\right)+\frac{x}{\tan x}<3+\frac{1}{60}x^4+\frac{960\pi-\pi^4-2880}{15\pi^6}x^{6},
\end{align}
where the constants $\frac{1}{504}$ and $(960\pi-\pi^4-2880)/(15\pi^6)$
are the best possible.

The formula  \eqref{Thm-Huygens-type-inequality2Key} motivated us to observe Theorem \ref{Huygens-type-inequality-proof2}.
\begin{theorem}\label{Huygens-type-inequality-proof2}
Let $N\geq1$ be an integer. Then for $0<x<\pi/2$, we have
\begin{align}\label{Huygens-inequality-generalization}
&3+\sum_{j=2}^{N}\frac{(2^{2j}-4){|B_{2j}|}}{(2j)!}x^{2j}+\rho_{N}x^{2N+1}\tan x<2\left(\frac{x}{\sin x}\right)+\frac{x}{\tan x}\nonumber\\
&\qquad\qquad\qquad\qquad\qquad\qquad<3+\sum_{j=2}^{N}\frac{(2^{2j}-4){|B_{2j}|}}{(2j)!}x^{2j}+\varrho_{N}x^{2N+1}\tan x
\end{align}
with the best possible constants
\begin{align}\label{Huygens-inequality-generalizationConstants}
\rho_N=0\quad\text{and}\quad \varrho_N=\frac{4(2^{2N}-1){|B_{2N+2}|}}{(2N+2)!}.
\end{align}
\end{theorem}

\begin{proof}
By \eqref{Thm-Huygens-type-inequality2Key}, for $\rho_N=0$, the first inequality in \eqref{Huygens-inequality-generalization} holds. We now prove the second inequality in \eqref{Huygens-inequality-generalization} with $\varrho_N=4(2^{2N}-1){|B_{2N+2}|}/(2N+2)!$.
Using \eqref{tan-expansion} and \eqref{Thm-Huygens-type-inequality2Key}, we  find
\begin{align}\label{Huygens-inequality-generalizationFind}
&\frac{4(2^{2N}-1){|B_{2N+2}|}}{(2N+2)!}x^{2N+1}\tan x-\left(2\left(\frac{x}{\sin x}\right)+\frac{x}{\tan x}-3-\sum_{j=2}^{N}\frac{(2^{2j}-4){|B_{2j}|}}{(2j)!}x^{2j}\right)\nonumber\\
&\quad=\frac{4(2^{2N}-1){|B_{2N+2}|}}{(2N+2)!}x^{2N+1}\sum_{k=1}^{\infty}\frac{2^{2k}(2^{2k}-1)|B_{2k}|}{(2k)!}x^{2k-1}-\sum_{k=N+1}^{\infty}\frac{(2^{2k}-4)|B_{2k}|}{(2k)!}x^{2k}\nonumber\\
&\quad=\sum_{k=N+2}^{\infty}\left\{\frac{4(2^{2N}-1){|B_{2N+2}|}}{(2N+2)!}\frac{2^{2k-2N}(2^{2k-2N}-1)|B_{2k-2N}|}{(2k-2N)!}-\frac{(2^{2k}-4)|B_{2k}|}{(2k)!}\right\}x^{2k}.
\end{align}

We claim that  for $k\geq N+2$,
\begin{align}\label{Thm-Huygens-type-inequality2Claim}
\frac{4(2^{2N}-1){|B_{2N+2}|}}{(2N+2)!}\frac{2^{2k-2N}(2^{2k-2N}-1)|B_{2k-2N}|}{(2k-2N)!}>\frac{(2^{2k}-4)|B_{2k}|}{(2k)!}.
\end{align}
Using the inequality \eqref{eq:1.11}, it is sufficient to prove that
\begin{align*}
\frac{4(2^{2N}-1)\cdot2}{\left(2\pi\right)^{2N+2}}\frac{2^{2k-2N}(2^{2k-2N}-1)\cdot2}{\left(2\pi\right)^{2k-2N}}> \frac{(2^{2k}-4)\cdot2}{\left(2\pi\right)^{2k}\left(1-2^{1-2k}\right)},\qquad k\geq N+2,
\end{align*}
which can be rearranged as
\begin{align*}
\left(1-\frac{1}{2^{2N}}\right)\left(\frac{2^{2k}}{2^{2N}}-1\right)>\frac{\pi^2}{2}\left(1-\frac{2}{2^{2k}-2}\right),\qquad k\geq N+2.
\end{align*}
Noting that $\pi^2/2<5$, it is enough to prove the following inequality:
\begin{align*}
\left(1-\frac{1}{2^{2N}}\right)\left(\frac{2^{2k}}{2^{2N}}-1\right)>5\left(1-\frac{2}{2^{2k}-2}\right),\qquad k\geq N+2,
\end{align*}
which can be written as
\begin{align*}
\left(1-\frac{1}{2^{2N}}\right)\frac{2^{2k}}{2^{2N}}+\frac{1}{2^{2N}}+\frac{10}{2^{2k}-2}>6,\qquad k\geq N+2.
\end{align*}
It is enough to prove the following inequality:
\begin{align}\label{Thm-Huygens-type-inequality2Claimenough}
\left(1-\frac{1}{2^{2N}}\right)\frac{2^{2k}}{2^{2N}}+\frac{1}{2^{2N}}>6,\qquad k\geq N+2.
\end{align}
Clearly,
\begin{align*}
\left(1-\frac{1}{2^{2N}}\right)\frac{2^{2k}}{2^{2N}}+\frac{1}{2^{2N}}\geq\left(1-\frac{1}{2^{2N}}\right)\frac{2^{2N+4}}{2^{2N}}+\frac{1}{2^{2N}}=16-\frac{15}{2^{2N}},\qquad k\geq N+2.
\end{align*}
In order to prove \eqref{Thm-Huygens-type-inequality2Claimenough}, it
suffices to show that
\begin{align*}
16-\frac{15}{2^{2N}}>6,\qquad N\geq1,
\end{align*}
that is,
\begin{align*}
2^{2N+1}>3,\qquad N\geq1.
\end{align*}
Obviously, the last inequality holds.   This proves the claim \eqref{Thm-Huygens-type-inequality2Claim}.
From \eqref{Huygens-inequality-generalizationFind}, we obtain the second inequality in \eqref{Huygens-inequality-generalization} with $\varrho_N=4(2^{2N}-1){|B_{2N+2}|}/(2N+2)!$.

Write \eqref{Huygens-inequality-generalization} as
\begin{align*}
\rho_N<\frac{2\left(\frac{x}{\sin x}\right)+\frac{x}{\tan x}-3-\sum_{j=2}^{N}\frac{(2^{2j}-4){|B_{2j}|}}{(2j)!}x^{2j}}{x^{2N+1}\tan x}<\varrho_N.
\end{align*}
We find
\begin{align*}
\lim_{x\to\frac{\pi}{2}}\frac{2\left(\frac{x}{\sin x}\right)+\frac{x}{\tan x}-3-\sum_{j=2}^{N}\frac{(2^{2j}-4){|B_{2j}|}}{(2j)!}x^{2j}}{x^{2N+1}\tan x}=0
\end{align*}
and
\begin{align*}
\lim_{x\to0}\frac{2\left(\frac{x}{\sin x}\right)+\frac{x}{\tan x}-3-\sum_{j=2}^{N}\frac{(2^{2j}-4){|B_{2j}|}}{(2j)!}x^{2j}}{x^{2N+1}\tan x}=\frac{4(2^{2N}-1){|B_{2N+2}|}}{(2N+2)!}.
\end{align*}
Hence, the inequality \eqref{Huygens-inequality-generalization} holds with the best possible constants given in \eqref{Huygens-inequality-generalizationConstants}. The proof of Theorem \ref{Huygens-type-inequality-proof2} is complete.
\end{proof}
\begin{remark}
 For $0<|x|<\pi/2$, we have
\begin{align}\label{Thm-Huygens-type-inequality1}
3+ax^3\tan x<2\left(\frac{x}{\sin x}\right)+\frac{x}{\tan x}<3+bx^3\tan x
\end{align}
with the best possible constants
\begin{align}\label{Thm-Huygens-type-inequality1best}
a=0\quad \text{and}\quad b=\frac{1}{60}.
\end{align}
There is no strict comparison between  the two upper bounds in \eqref{Best-Huygens-type-inequality-xN=1} and \eqref{Thm-Huygens-type-inequality1}.
\end{remark}

\vskip 8mm

\section{The Papenfuss-Bach inequality}

Papenfuss \cite{Papenfuss765} proposed the following  problem:

Prove that
\begin{equation}\label{sec-tan}
x\sec^{2} x-\tan x\leq\frac{8\pi^{2}x^3}{\big(\pi^2-4x^2\big)^2},\qquad 0\leq x<\pi/2.
\end{equation}
Bach \cite{Bach62} proved the  inequality \eqref{sec-tan} and obtained a further result as follows:
\begin{equation}\label{Bachsec-tan}
x\sec^{2} x-\tan x\leq\frac{(2\pi^4/3)x^3}{\big(\pi^2-4x^2\big)^2},\qquad 0\leq x<\pi/2.
\end{equation}
Ge \cite[Theorem 1.3]{Ge2012} presented a lower bound in \eqref{Bachsec-tan} and proved that
\begin{equation}\label{Gesec-tan}
\frac{64x^3}{\big(\pi^2-4x^2\big)^2}<x\sec^{2} x-\tan x\leq\frac{(2\pi^4/3)x^3}{\big(\pi^2-4x^2\big)^2},\qquad 0\leq x<\pi/2,
\end{equation}
where the constants $64$ and $2\pi^4/3$ are the best possible.

Sun and Zhu  \cite[Theorem 1.5]{Sun-Zhu869261} obtained better bounds for the Papenfuss-Bach inequality:
\begin{equation}\label{SunZhusec-tan}
\frac{\frac{2\pi^4}{3}x^3+\left(\frac{8\pi^4}{15}-\frac{16\pi^2}{3}\right)x^5}{\big(\pi^2-4x^2\big)^2}<x\sec^{2} x-\tan x<\frac{\frac{2\pi^4}{3}x^3+\left(\frac{256}{\pi^2}\cdot(\frac{513}{511})-\frac{8\pi^2}{3}\right)x^5}{\big(\pi^2-4x^2\big)^2}
\end{equation}
for $0< x<\pi/2$.

Also in \cite{Sun-Zhu869261}, Sun and Zhu posed the following
\begin{open}\label{Open-sec-tan}
Let $0 <x<\pi/2$. Then
\begin{equation}\label{SunZhusec-tanOpen}
\frac{\frac{2\pi^4}{3}x^3+\left(\frac{8\pi^4}{15}-\frac{16\pi^2}{3}\right)x^5}{\big(\pi^2-4x^2\big)^2}<x\sec^{2} x-\tan x<\frac{\frac{2\pi^4}{3}x^3+\left(\frac{256}{\pi^2}-\frac{8\pi^2}{3}\right)x^5}{\big(\pi^2-4x^2\big)^2}
\end{equation}
hold, where $\frac{8\pi^4}{15}-\frac{16\pi^2}{3}$ and $\frac{256}{\pi^2}-\frac{8\pi^2}{3}$ are the best constants in \eqref{SunZhusec-tanOpen}.
\end{open}
In this section, we present a series
representation of the remainder in the expansion  for $t\sec^{2}t-\tan t$. Based on this representation, we establish new bounds for $x\sec^{2} x-\tan x$. We also answer the open problem \ref{Open-sec-tan}.

It follows from \cite[p. 44]{Gradshteyn2000} that
\begin{equation*}
\sec^{2}\frac{\pi x}{2}=\frac{4}{\pi^2}\sum_{k=1}^{\infty}\left\{\frac{1}{(2k-1-x)^{2}}+\frac{1}{(2k-1+x)^{2}}\right\}.
\end{equation*}
Replacement of $x$ by $2t/\pi$ yields
\begin{equation}\label{sec-series}
\sec^{2}t=\frac{4}{\pi^2}\sum_{k=1}^{\infty}\left\{\frac{1}{(2k-1-\frac{2t}{\pi})^{2}}+\frac{1}{(2k-1+\frac{2t}{\pi})^{2}}\right\}.
\end{equation}
From \eqref{tan-expansion-re} and \eqref{sec-series}, we have
\begin{equation}\label{sec-tan-expansion}
t\sec^{2}t-\tan t=\frac{64t^3}{\pi^4}\sum_{k=1}^{\infty}\frac{1}{(2k-1)^{4}\left(1-\big(\frac{2t}{\pi(2k-1)}\big)^2\right)^2}.
\end{equation}
Using \eqref{identity1/(1-q)2}
and \eqref{odd-series-Expansion},
we obtain from \eqref{sec-tan-expansion} the series
representation of the remainder in the expansion  for $\sec^{2}t-\tan t/t$:
\begin{align}\label{sec-tan+RN}
&t\sec^{2}t-\tan t\nonumber\\
&=\frac{64t^3}{\pi^4}\sum_{k=1}^{\infty}\frac{1}{(2k-1)^{4}}\left(\sum_{j=1}^{N-1}j\bigg(\frac{2t}{\pi(2k-1)}\bigg)^{2j-2}+\frac{N\big(\frac{2t}{\pi(2k-1)}\big)^{2N-2}}{1-\big(\frac{2t}{\pi(2k-1)}\big)^2}+\frac{\big(\frac{2t}{\pi(2k-1)}\big)^{2N}}{\Big(1-\big(\frac{2t}{\pi(2k-1)}\big)^2\Big)^2}\right)\nonumber\\
&=\sum_{j=1}^{N-1}\frac{2j\cdot2^{2j+2}(2^{2j+2}-1)|B_{2j+2}|}{(2j+2)!}t^{2j+1}+\kappa_N(t),
\end{align}
where
\begin{align}\label{sec-tan-RN}
\kappa_N(t)&=\frac{N\cdot2^{2N+4}t^{2N+1}}{\pi^{2N}}\sum_{k=1}^{\infty}\frac{1}{(2k-1)^{2N}\big(\pi^2(2k-1)^2-4t^2\big)}\nonumber\\
&\quad+\frac{2^{2N+6}t^{2N+3}}{\pi^{2N}}\sum_{k=1}^{\infty}\frac{1}{(2k-1)^{2N}\big(\pi^2(2k-1)^2-4t^2\big)^2}.
\end{align}

\begin{theorem}\label{Thm-remainder-sectan}
Let $N\geq1$ be an integer. Then for $0<t<\pi/2$, we have
\begin{align}\label{sec-tan+RNreobtain}
L_N(t)&<t\sec^{2}t-\tan t-\sum_{j=1}^{N-1}\frac{2j\cdot2^{2j+2}(2^{2j+2}-1)|B_{2j+2}|}{(2j+2)!}t^{2j+1}\nonumber\\
&\qquad\qquad\qquad\quad-\frac{N\cdot2^{2N+4}t^{2N+1}}{\pi^{2N}(\pi^2-4t^2)}-\frac{2^{2N+6}t^{2N+3}}{\pi^{2N}(\pi^2-4t^2)^2}<M_N(t),
\end{align}
where
\begin{align*}
L_N(t)&=\frac{N\cdot2^{2N+4}t^{2N+1}}{\pi^{2N+2}}\left\{\frac{(2^{2N+2}-1)\pi^{2N+2}|B_{2N+2}|}{2\cdot(2N+2)!}-1\right\}\\
&\quad+\frac{2^{2N+6}t^{2N+3}}{\pi^{2N+4}}\left\{\frac{(2^{2N+4}-1)\pi^{2N+4}|B_{2N+4}|}{2\cdot(2N+4)!}-1\right\}
\end{align*}
and
\begin{align*}
M_N(t)=\frac{N\cdot2^{2N+2}t^{2N+1}}{\pi^{2N+2}}\sum_{k=2}^{\infty}\frac{1}{(2k-1)^{2N}k(k-1)}+\frac{2^{2N+2}t^{2N+3}}{\pi^{2N+4}}\sum_{k=2}^{\infty}\frac{1}{(2k-1)^{2N}k^2(k-1)^2}.
\end{align*}
\end{theorem}

\begin{proof}
Write \eqref{sec-tan+RN} as
\begin{align}\label{sec-tan+RNre}
t\sec^{2}t-\tan t&=\sum_{j=1}^{N-1}\frac{2j\cdot2^{2j+2}(2^{2j+2}-1)|B_{2j+2}|}{(2j+2)!}t^{2j+1}\nonumber\\
&\quad+\frac{N\cdot2^{2N+4}t^{2N+1}}{\pi^{2N}(\pi^2-4t^2)}+\frac{N\cdot2^{2N+4}t^{2N+1}}{\pi^{2N}}I_N(t)\nonumber\\
&\quad+\frac{2^{2N+6}t^{2N+3}}{\pi^{2N}(\pi^2-4t^2)^2}+\frac{2^{2N+6}t^{2N+3}}{\pi^{2N}}J_N(t),
\end{align}
where
\begin{align*}
I_N(t)&=\sum_{k=2}^{\infty}\frac{1}{(2k-1)^{2N}\big(\pi^2(2k-1)^2-4t^2\big)}
\end{align*}
and
\begin{align*}
J_N(t)=\sum_{k=2}^{\infty}\frac{1}{(2k-1)^{2N}\big(\pi^2(2k-1)^2-4t^2\big)^2}.
\end{align*}
Obviously, $I_{N}(t)$ and  $J_{N}(t)$  are both strictly increasing for $t\in (0, \pi/2)$. We then obtain from \eqref{sec-tan+RNre} that
\begin{align*}
&\frac{N\cdot2^{2N+4}t^{2N+1}}{\pi^{2N}}I_N(0)+\frac{2^{2N+6}t^{2N+3}}{\pi^{2N}}J_N(0)\nonumber\\
&\quad<t\sec^{2}t-\tan t-\sum_{j=1}^{N-1}\frac{2j\cdot2^{2j+2}(2^{2j+2}-1)|B_{2j+2}|}{(2j+2)!}t^{2j+1}
-\frac{N\cdot2^{2N+4}t^{2N+1}}{\pi^{2N}(\pi^2-4t^2)}-\frac{2^{2N+6}t^{2N+3}}{\pi^{2N}(\pi^2-4t^2)^2}\nonumber\\
&\quad<\frac{N\cdot2^{2N+4}t^{2N+1}}{\pi^{2N}}I_N\left(\frac{\pi}{2}\right)+\frac{2^{2N+6}t^{2N+3}}{\pi^{2N}}J_N\left(\frac{\pi}{2}\right).
\end{align*}

Direct computations yield
\begin{align*}
L_N(t)&=\frac{N\cdot2^{2N+4}t^{2N+1}}{\pi^{2N}}I_N(0)+\frac{2^{2N+6}t^{2N+3}}{\pi^{2N}}J_N(0)\nonumber\\
&=\frac{N\cdot2^{2N+4}t^{2N+1}}{\pi^{2N+2}}\sum_{k=2}^{\infty}\frac{1}{(2k-1)^{2N+2}}+\frac{2^{2N+6}t^{2N+3}}{\pi^{2N+4}}\sum_{k=2}^{\infty}\frac{1}{(2k-1)^{2N+4}}\\
&=\frac{N\cdot2^{2N+4}t^{2N+1}}{\pi^{2N+2}}\left\{\frac{(2^{2N+2}-1)\pi^{2N+2}|B_{2N+2}|}{2\cdot(2N+2)!}-1\right\}\\
&\quad+\frac{2^{2N+6}t^{2N+3}}{\pi^{2N+4}}\left\{\frac{(2^{2N+4}-1)\pi^{2N+4}|B_{2N+4}|}{2\cdot(2N+4)!}-1\right\}
\end{align*}
and
\begin{align*}
M_N(t)&=\frac{N\cdot2^{2N+4}t^{2N+1}}{\pi^{2N}}I_N\left(\frac{\pi}{2}\right)+\frac{2^{2N+6}t^{2N+3}}{\pi^{2N}}J_N\left(\frac{\pi}{2}\right)\nonumber\\
&=\frac{N\cdot2^{2N+2}t^{2N+1}}{\pi^{2N+2}}\sum_{k=2}^{\infty}\frac{1}{(2k-1)^{2N}k(k-1)}+\frac{2^{2N+2}t^{2N+3}}{\pi^{2N+4}}\sum_{k=2}^{\infty}\frac{1}{(2k-1)^{2N}k^2(k-1)^2}.
\end{align*}
 The proof of Theorem \ref{Thm-remainder-sectan} is complete.
\end{proof}

With the evaluations
\begin{align*}
\sum_{k=2}^{\infty}\frac{1}{(2k-1)^{4}k(k-1)}=9-\frac{\pi^4}{24}-\frac{\pi^2}{2}
\end{align*}
and
\begin{align*}
\sum_{k=2}^{\infty}\frac{1}{(2k-1)^{4}k^2(k-1)^2}=-59+\frac{13\pi^2}{3}+\frac{\pi^4}{6},
\end{align*}
the choice $N=2$ in \eqref{sec-tan+RNreobtain} yields
\begin{equation}\label{Chensec-tan}
\frac{P(x)}{\big(\pi^2-4x^2\big)^2}<x\sec^{2} x-\tan x<\frac{Q(x)}{\big(\pi^2-4x^2\big)^2},\quad 0<x<\frac{\pi}{2},
\end{equation}
where
\begin{align*}
P(x)&=\frac{2\pi^4}{3}x^3+\frac{8\pi^2(\pi^2-10)}{15}x^5+\frac{2(322560+1680\pi^4-672\pi^{6}+17\pi^{8})}{315\pi^4}x^7\\
&\quad+\frac{16(168-17\pi^2)}{315}x^9+\frac{32(17\pi^8-161280)}{315\pi^8}x^{11}
\end{align*}
and
\begin{align*}
Q(x)&=\frac{2\pi^4}{3}x^3+\frac{32(156-6\pi^2-\pi^4)}{3\pi^2}x^5+\frac{64(-657+37\pi^2+3\pi^4)}{3\pi^4}x^7\\
&\quad+\frac{512(285-19\pi^2-\pi^4)}{3\pi^6}x^9+\frac{512(-354+26\pi^2+\pi^4)}{3\pi^8}x^{11}.
\end{align*}
 The inequality  \eqref{Chensec-tan} is an improvement on the inequality \eqref{SunZhusec-tan}.

 \begin{remark}
 In fact,  the lower bound in  \eqref{Chensec-tan} is larger than the one in  \eqref{SunZhusec-tanOpen}, and   the upper bound in  \eqref{Chensec-tan} is smaller  than the one in  \eqref{SunZhusec-tanOpen}.    Hence, the inequality \eqref{SunZhusec-tanOpen} holds true.
If we write \eqref{SunZhusec-tanOpen} as
\begin{equation*}
\frac{8\pi^4}{15}-\frac{16\pi^2}{3}<\frac{(x\sec^{2} x-\tan x)(\pi^2-4x^2)^2-\frac{2\pi^4}{3}x^3}{x^5}<\frac{256}{\pi^2}-\frac{8\pi^2}{3},
\end{equation*}
we find that
\begin{equation*}\label{SunZhusec-tanOpenre}
\lim_{x\to0}\frac{(x\sec^{2} x-\tan x)(\pi^2-4x^2)^2-\frac{2\pi^4}{3}x^3}{x^5}=\frac{8\pi^4}{15}-\frac{16\pi^2}{3}
\end{equation*}
and
\begin{equation*}\label{SunZhusec-tanOpenre}
\lim_{x\to \pi/2}\frac{(x\sec^{2} x-\tan x)(\pi^2-4x^2)^2-\frac{2\pi^4}{3}x^3}{x^5}=\frac{256}{\pi^2}-\frac{8\pi^2}{3}.
\end{equation*}
Hence, the inequality \eqref{SunZhusec-tanOpen} holds for $0<x<\pi/2$, and the  constants $\frac{8\pi^4}{15}-\frac{16\pi^2}{3}$ and $\frac{256}{\pi^2}-\frac{8\pi^2}{3}$ are the best possible.
\end{remark}

\vskip 8mm
\section{A double inequality for the remainder in the expansion for $\sec x$}

Let $S_n(x)$  denote
\begin{align}\label{tan-expansion-remainder}
S_n(x)=\sum_{k=1}^{n}\frac{2^{2k}(2^{2k}-1)|B_{2k}|}{(2k)!}x^{2k-1},\qquad |x|<\frac{\pi}{2}.
\end{align}
 By using induction, Chen and Qi \cite{Chen-Qi} (see also \cite{Zhao499--506}) established a double inequality for the difference $\tan x-S_n(x)$:
\begin{equation}\label{mainr}
\frac{2^{2n+2}(2^{2n+2}-1){|B_{2n+2}|}}{(2n+2)!}x^{2n}\tan x
<\tan x-S_n(x)<\left(\frac2\pi\right)^{2n}x^{2n}\tan x
\end{equation}
 for $0<x<\pi/2$ and $n\in\mathbb{N}$, where the the constants
\begin{equation*}
\frac{2^{2n+2}(2^{2n+2}-1){|B_{2n+2}|}}{(2n+2)!}\quad\text{and}\quad \left(\frac2\pi\right)^{2n}
\end{equation*}
are the best possible.

It is well known \cite[p. 43]{Gradshteyn2000}  that
\begin{align}\label{sec-powerseries}
\sec x=\sum_{j=0}^{\infty}\frac{|E_{2j}|}{(2j)!}x^{2j},\qquad |x|<\frac{\pi}{2}.
\end{align}
Let $s_N(x)$  denote
\begin{align}\label{tan-expansion-remainder}
s_N(x)=\sum_{j=0}^{N-1}\frac{|E_{2j}|}{(2j)!}x^{2j},\qquad |x|<\frac{\pi}{2}.
\end{align}
In this section, we establish a double inequality for the difference $\sec x-s_N(x)$, which is an analogous result to \eqref{mainr} given by Theorem \ref{Thm-Inequality-Remainder-sec}.
\begin{theorem}\label{Thm-Inequality-Remainder-sec}
Let $N\geq0$ be an integer. Then for $0<x<\pi/2$, we have
\begin{align}\label{secx-inequality}
\frac{|E_{2N}|}{(2N)!}x^{2N-1}\tan x<\sec x-s_N(x)<\left(\frac{2}{\pi}\right)^{2N-1}x^{2N-1}\tan x,
\end{align}
where the constants $|E_{2N}|/(2N)!$ and $(2/\pi)^{2N-1}$ are the best possible.
\end{theorem}

\begin{proof}
By \eqref{tanx-expansion} and \eqref{sec-powerseries},
the left-hand side inequality \eqref{secx-inequality} can be written for $0<x<\pi/2$ as
\begin{align*}
\sum_{j=N}^{\infty}\frac{|E_{2N}|}{(2N)!}\frac{2^{2j-2N+2}(2^{2j-2N+2}-1){|B_{2j-2N+2}|}}{(2j-2N+2)!}x^{2j}<\sum_{j=N}^{\infty}\frac{|E_{2j}|}{(2j)!}x^{2j},
\end{align*}
or
\begin{align*}
\sum_{j=N+1}^{\infty}\left\{\frac{|E_{2N}|}{(2N)!}\frac{2^{2j-2N+2}(2^{2j-2N+2}-1){|B_{2j-2N+2}|}}{(2j-2N+2)!}-\frac{|E_{2j}|}{(2j)!}\right\}x^{2j}<0.
\end{align*}

We now prove that
\begin{align}\label{Thmproofright-Inequality-Remainder-seckey}
\frac{|E_{2N}|}{(2N)!}\frac{2^{2j-2N+2}(2^{2j-2N+2}-1){|B_{2j-2N+2}|}}{(2j-2N+2)!}<\frac{|E_{2j}|}{(2j)!},\qquad j\geq N+1.
\end{align}
Using \eqref{eq:1.11} and the following inequality (see \cite[p. 805]{abram}):
 \begin{eqnarray}\label{iNEQUALITY-En}
 \frac{4^{n+1}}{\pi^{2n+1}}\left(\frac{1}{1+3^{-1-2n}}\right)<\frac{\left|E_{2n}\right|}{(2n)!}< \frac{4^{n+1}}{\pi^{2n+1}},\qquad
n=0, 1, 2,\ldots,
\end{eqnarray}
 it suffices to show that
\begin{align*}
\frac{4^{N+1}}{\pi^{2N+1}}\frac{2^{2j-2N+2}(2^{2j-2N+2}-1)2}{\left(2\pi\right)^{2j-2N+2}\left(1-2^{1-2(j-N+1)}\right)}< \frac{4^{j+1}}{\pi^{2j+1}}\left(\frac{1}{1+3^{-1-2j}}\right),\qquad j\geq N+1,
\end{align*}
which can be rearranged as
\begin{align*}
\frac{8}{\pi^{2}}\frac{4^{j-N+1}-1}{4^{j-N+1}-2}< \frac{3^{2j+1}}{3^{2j+1}+1},
\end{align*}
\begin{align*}
\frac{8}{\pi^{2}}\left(1+\frac{1}{4^{j-N+1}-2}\right)< 1-\frac{1}{3^{2j+1}+1},
\end{align*}
\begin{align*}
\frac{8}{\pi^2(4^{j-N+1}-2)}+\frac{1}{3^{2j+1}+1}< 1-\frac{8}{\pi^2}.
\end{align*}
Noting that the sequence
\begin{align*}
\frac{8}{\pi^2(4^{j-N+1}-2)}+\frac{1}{3^{2j+1}+1}
\end{align*}
is strictly decreasing for $j\geq N+1$,  it is enough to prove the following inequality:
\begin{align*}
\frac{4}{7\pi^2}+\frac{1}{3^{2N+3}+1}< 1-\frac{8}{\pi^2},
\end{align*}
which can be rearranged as
\begin{align}\label{Thmproofleft-Inequality-Remainder-sec}
3^{2N+3}>\frac{60}{7\pi^2-60}=6.60267151\ldots.
\end{align}
Obviously, \eqref{Thmproofleft-Inequality-Remainder-sec} holds for all integers $N\geq0$. This proves \eqref{Thmproofright-Inequality-Remainder-seckey}.
Hence, the left-hand side inequality \eqref{secx-inequality}  holds.

By Theorem \ref{Thm-remainder-sec} and \eqref{tan-expansion},
the right-hand side inequality \eqref{secx-inequality} can be rearranged for $0<x<\pi/2$ as
\begin{align*}
\sum_{k=1}^{\infty}\frac{(-1)^{k+1}}{(2k-1)^{2N-1}\Big(\pi^2(2k-1)^2-4x^2\Big)}<\sum_{k=1}^{\infty}\frac{1}{\pi^2(2k-1)^{2}-4x^2},
\end{align*}
or
\begin{align}\label{Thmproofright-Inequality-Remainder-sec}
\sum_{k=2}^{\infty}\left\{1-\frac{(-1)^{k+1}}{(2k-1)^{2N-1}}\right\}\frac{1}{\pi^2(2k-1)^{2}-4x^2}>0.
\end{align}
Obviously, \eqref{Thmproofright-Inequality-Remainder-sec} holds. Hence, the right-hand side inequality \eqref{secx-inequality}  holds.

Write \eqref{secx-inequality} as
\begin{align*}
\frac{|E_{2N}|}{(2N)!}<\frac{\sec x-\sum_{j=0}^{N-1}\frac{|E_{2j}|}{(2j)!}x^{2j}}{x^{2N-1}\tan x}<\left(\frac{2}{\pi}\right)^{2N-1}.
\end{align*}
We find
\begin{align*}
\lim_{x\to0}\frac{\sec x-\sum_{j=0}^{N-1}\frac{|E_{2j}|}{(2j)!}x^{2j}}{x^{2N-1}\tan x}=\frac{|E_{2N}|}{(2N)!}
\end{align*}
and
\begin{align*}
\lim_{x\to\frac{\pi}{2}}\frac{\sec x-\sum_{j=0}^{N-1}\frac{|E_{2j}|}{(2j)!}x^{2j}}{x^{2N-1}\tan x}=\left(\frac{2}{\pi}\right)^{2N-1}.
\end{align*}
Hence, the inequality \eqref{secx-inequality} holds,  the constants $|E_{2N}|/(2N)!$ and $(2/\pi)^{2N-1}$ are the best possible. The proof of Theorem \ref{Thm-Inequality-Remainder-sec} is complete.
\end{proof}

\vskip 8mm
\begin{center}
{ Appendix:  A proof of \eqref{Sharp-Wilker-type-inequality-find2}}
\end{center}

For $N=0$ in \eqref{Sharp-Wilker-type-inequality-find2}, we find  that
\begin{align*}
\sum_{k=1}^{\infty}\frac{1}{\big(4k^2-1\big)^2}=\frac{\pi^2-8}{16} \quad \text{and}\quad \frac{1}{8}\psi'\left(\frac{1}{2}\right)-\frac{1}{2}=\frac{\pi^2-8}{16}.
\end{align*}
This shows that the formula \eqref{Sharp-Wilker-type-inequality-find2} holds  for $N=0$.
\par
Now we assume  that the formula \eqref{Sharp-Wilker-type-inequality-find2} holds
for some $N \in \mathbb{N}_0:=\mathbb{N}\cup\{0\}$. Then,  for $N\mapsto N+1$ in \eqref{Sharp-Wilker-type-inequality-find2}, by using the induction hypothesis and the  following relation:
\begin{align*}
\psi'(z+1)=\psi'(z)-\frac{1}{z^2},
\end{align*}
we have
\begin{align*}
\sum_{k=N+2}^{\infty}\frac{1}{\big(4k^2-1\big)^2}&=\sum_{k=N+1}^{\infty}\frac{1}{\big(4k^2-1\big)^2}-\frac{1}{\big(4(N+1)^2-1\big)^2}\\
&=\frac{1}{8}\psi'\left(N+\frac{1}{2}\right)-\frac{N+1}{2(2N+1)^2}-\frac{1}{\big(4(N+1)^2-1\big)^2}\\
&=\frac{1}{8}\psi'\left(N+\frac{1}{2}\right)-\frac{1}{8(N+\frac{1}{2})^2}-\frac{N+2}{2\big(2N+3\big)^2}\\
&=\frac{1}{8}\psi'\left(N+\frac{3}{2}\right)-\frac{N+2}{2\big(2N+3\big)^2}.
\end{align*}
Thus, by  the principle of  mathematical induction,  the
formula  \eqref{Sharp-Wilker-type-inequality-find2} holds for all $N \in \mathbb{N}_0$.

 \vskip 4mm

\section*{Acknowledgement} Some computations in this paper
were performed using Maple software.

\enddocument